\documentclass[draft, 12pt]{article}
\usepackage{amsmath,amsthm,amssymb}

\usepackage[top=30truemm, bottom=30truemm, left=25truemm, right=25truemm]{geometry}

\newtheorem{defi}{Definition}[section]

\newtheorem{lem}[defi]{Lemma}
\newtheorem{prop}[defi]{Proposition}
\newtheorem{thm}[defi]{Theorem}
\newtheorem{cor}[defi]{Corollary}
\newtheorem{ex}[defi]{Example}

\usepackage[dvipdfmx, colorlinks=true]{hyperref}
\hypersetup{final=true}

\DeclareMathOperator{\rank}{rank}
\DeclareMathOperator{\Span}{Span}

\usepackage{ytableau}

\title{
An Algebra Associated with a Flag in
a Subspace Lattice over a Finite Field
and the Quantum Affine Algebra $U_q(\widehat{\mathfrak{sl}}_2)$
}
\author{Yuta Watanabe}
\date{\today}

\begin{document}
\maketitle

\begin{abstract}
In this paper, we introduce 
an algebra $\mathcal{H}$ from a subspace lattice
with respect to a fixed flag
which contains its incidence algebra as a proper subalgebra.
We then establish a relation between the algebra $\mathcal{H}$ and
the quantum affine algebra $U_{q^{1/2}}(\widehat{\mathfrak{sl}}_2)$,
where $q$ denotes the cardinality of the base field.
It is an extension of the well-known relation 
between the incidence algebra of a subspace lattice and 
the quantum algebra $U_{q^{1/2}}(\mathfrak{sl}_2)$.
We show that there exists an algebra homomorphism
from $U_{q^{1/2}}(\widehat{\mathfrak{sl}}_2)$ to $\mathcal{H}$
and that
any irreducible module for $\mathcal{H}$ is irreducible
as an $U_{q^{1/2}}(\widehat{\mathfrak{sl}}_2)$-module.

\bigskip

\end{abstract}

\section{Introduction}\label{introduction}

By a \emph{subspace lattice}, also known as a \emph{projective geometry},
we mean the partially ordered set (poset) of
all subspaces of a finite-dimensional vector space over a finite field,
where the ordering is given by inclusion.
In the field of combinatorics,
subspace lattices are regarded as $q$-analogs of Boolean lattices
and therefore they have been studied from many combinatorial points of view, such as
Grassmann codes and Grassmann graphs.
On the other hand,
the \emph{quantum affine algebras} $U_q(\widehat{\mathfrak{sl}}_2)$ are 
Hopf algebras that are $q$-deformations of the universal enveloping algebra
of the affine Lie algebra $\widehat{\mathfrak{sl}}_2$ and
their representations are developed in \cite[Section~5]{Ch} as trigonometric solutions of the quantum Yang--Baxter equation.
Recently, the author succeeded in \cite{W} in establishing a relation between
an algebra associated with a subspace lattice and the quantum affine algebras $U_q(\widehat{\mathfrak{sl}}_2)$
as an extension of the well-known relation
between the incidence algebra of a subspace lattice and 
the quantum algebras $U_q(\mathfrak{sl}_2)$.
In this paper, we introduce 
another algebra and establish its relation to the quantum affine algebra  $U_q(\widehat{\mathfrak{sl}}_2)$
which is in some sense the opposite extreme to that obtained in \cite{W}.

Here we briefly recall the known facts.
See \cite{Te90}, \cite{Te03} and \cite{W} for more detail.
Let $H$ denote an $N$-dimensional vector space
over a finite field $\mathbb{F}_q$ of $q$ elements
and let $P$ denote the subspace lattice consisting of all subspaces of $H$.
From the poset structure of $P$, we define the \emph{lowering matrix} $L$ indexed by $P$
whose $(x,y)$-entry is $1$ if $y$ covers $x$ and $0$ otherwise for $x,y \in P$.
Similarly, we define the \emph{raising matrix} $R$ indexed by $P$
whose $(x,y)$-entry is $1$ if $x$ covers $y$ and $0$ otherwise for $x,y \in P$.
The poset $P$ has the grading which is a partition of $P$
into nonempty sets
\begin{align*}
P_i = \lbrace y \in P \mid \dim y = i \rbrace
&&
(0 \le i \le N).
\end{align*}
From this grading structure, for $0 \le i \le N$,
we define the \emph{$i$-th projection matrix} $E_i^\star$ 
by the diagonal matrix indexed by $P$
whose $(x,x)$-entry is $1$ if $x \in P_i$ and $0$ otherwise for $x \in P$.
By the \emph{incidence algebra}, we mean the complex matrix algebra generated by
the above three kinds of matrices $L$, $R$ and $E_i^\star$, where $0 \le i \le N$.
It is known that there exists a surjective algebra homomorphism 
from the quantum algebra $U_{q^{1/2}}(\mathfrak{sl}_2)$ to the incidence algebra.
Moreover, it is also known that
any irreducible module for the incidence algebra induces 
an irreducible $U_{q^{1/2}}(\mathfrak{sl}_2)$-module of type $1$.

In our previous paper \cite{W},
we extended the algebra homomorphism as follows.
Let us fix one subspace $x \in P$
with $0 < \dim x < N$ and
consider the following new ``rectangle'' partition of $P$
with respect to $x$:
\begin{equation}
P_{i,j} = 
\lbrace y \in P \mid \dim y = i+j, \dim(y \cap x) = i \rbrace
\label{eq:Pi}
\end{equation}
for $0 \le i \le \dim x$ and for $0 \le j \le N - \dim x$.
Remark that this is a refinement of the grading.
Then define the new projection matrices with respect to this partition
and define the complex matrix algebra generated by
the lowering, raising matrices and these new projection matrices.
By the construction, this new algebra contains the incidence algebra as its subalgebra.
Then it is shown in \cite{W} that there exists an algebra homomorphism 
from the quantum affine algebra $U_{q^{1/2}}(\widehat{\mathfrak{sl}}_2)$ to the new algebra,
which extends the above algebra homomorphism
from $U_{q^{1/2}}(\mathfrak{sl}_2)$ to the incidence algebra.
Moreover it is also shown in \cite{W} that 
any irreducible module for the new algebra 
induces an irreducible $U_{q^{1/2}}(\widehat{\mathfrak{sl}}_2)$-module of type $(1,1)$
which is more precisely a tensor product of two evaluation modules.

Now we summarize the main results of this paper.
We fix a (full) flag $\lbrace x_i \rbrace_{i = 0}^N$ on $H$ instead of the subspace $x \in P$,
and consider the following new ``hyper-cubic'' partition of $P$
with respect to $\lbrace x_i \rbrace_{i = 0}^N$:
\begin{equation}
P_\mu =
\lbrace y \in P \mid \dim(y \cap x_i) = \mu_1 + \mu_2 + \cdots + \mu_i \; (1 \le i \le N)\rbrace
\label{eq:Pi1}
\end{equation}
for $\mu = (\mu_1, \mu_2, \ldots, \mu_N) \in \lbrace 0,1 \rbrace^N$.
Then for $\mu \in \lbrace 0,1 \rbrace^N$,
we define the \emph{projection matrix} $E_\mu^*$ 
by the diagonal matrix indexed by $P$
whose $(y,y)$-entry is $1$ if $y \in P_\mu$ and $0$ otherwise for $y \in P$.
We next define the complex matrix algebra $\mathcal{H}$ generated by
the lowering, raising matrices and these new projection matrices
$E_\mu^*$, where $\mu \in \lbrace 0,1 \rbrace^N$.
By the construction, the algebra $\mathcal{H}$ contains the incidence algebra as its subalgebra.
We prove that there exists an algebra homomorphism 
from the quantum affine algebra $U_{q^{1/2}}(\widehat{\mathfrak{sl}}_2)$ to the algebra $\mathcal{H}$,
which again extends the above algebra homomorphism
from $U_{q^{1/2}}(\mathfrak{sl}_2)$ to the incidence algebra.
Moreover it is also proved that 
any irreducible module for the algebra $\mathcal{H}$ 
induces an irreducible $U_{q^{1/2}}(\widehat{\mathfrak{sl}}_2)$-module of type $(1,1)$
which is more precisely a tensor product of evaluation modules of dimension $2$.
Our main results are Theorems \ref{support} and \ref{isomorphic}.
To prove the main theorems,
we classify all the $\mathcal{H}$-modules up to isomorphism
and determine the multiplicities appearing in the standard module.

Seen from the viewpoint
of the action of the general linear group $\mathrm{GL}(N,q)$ on the subspace lattice $P$,
we may say the results of this paper are ``opposite'' to those obtained in our previous paper \cite{W}.
(In this paper, however, we will not take this point of view in any essential way.
We refer the reader to \cite{Du} for this viewpoint.)
Indeed, the partitions \eqref{eq:Pi} and \eqref{eq:Pi1}
turn out to be the orbits of
maximal and minimal parabolic subgroups of $\mathrm{GL}(N,q)$, respectively.
More precisely,
the corresponding subgroups
stabilize the fixed subspace $x$
and the fixed flag $\lbrace x_i \rbrace_{i = 0}^N$, respectively.

It is worth pointing out that our proofs involve
a natural and intrinsic combinatorial characterization
of the subspace lattice,
while the method used in our previous paper \cite{W}
is rather oriented towards Lie theory and the representation theory of quantum groups.
In this paper, 
we fix a basis $v_1, v_2, \ldots, v_N$ for $H$
such that $x_i$ is spanned by $v_1, v_2, \ldots, v_i$ for $1 \le i \le N$.
With respect to the basis,
we identify each subspace in $P$
with a certain matrix whose entries are in the base field $\mathbb{F}_q$.

Then,
we relate these matrices to
classical combinatorial objects, such as
Ferrers boards, rook placements and inversion numbers,
and interpret algebraic properties of subspaces in terms of these matrices
(and moreover, of other combinatorial objects above).
Almost all the problems which we concern in this paper
arrive at problems in such classical combinatorial fields.
This type of argument is motivated by
Delsarte \cite{De}
and the technique used in this paper is a kind of
a generalized version of that in \cite{De}.

Comparing the partitions \eqref{eq:Pi} and \eqref{eq:Pi1} again,
one may ask whether same kinds of results can still be obtained
if we take a more general partition, 
which is defined by replacing a subspace or a full flag  by a general flag.
We will not develop this point here because
the required computation is expected to be far more complicated.
However we emphasize that we have done for the two extremal and the most essential cases,
and conjecture that similar results still hold in the general case.

We organize this paper as follows.
In Section \ref{hyper-cubic},
we recall the basic notation and introduce a hyper-cubic structure in a subspace lattice.
In Section \ref{Ferrers},
we recall some notation on Ferrers boards, rook placements and inversion numbers
which is used in this paper.
In Sections \ref{matrixrep} and \ref{sec:number},
we introduce a matrix representation of $P$ and interpret 
some properties of matrices in terms of 
rook placements and inversion numbers.
In Sections \ref{sec:H} and \ref{sec:strH},
we introduce the main object of this paper, the algebra $\mathcal{H}$,
and discuss the structure of it.
In Sections \ref{LRactiononV}, \ref{LRRLactiononV}, \ref{sec:kappa} and \ref{sec:Hmod},
we discuss the $\mathcal{H}$-action on the standard module and
classify all the irreducible $\mathcal{H}$-modules up to isomorphism.
In Section \ref{Uq(sl2hat)},
for the convenience of the reader,
we repeat the relevant material,
including the definition of the quantum affine algebra $U_q(\widehat{\mathfrak{sl}}_2)$,
from \cite{Ch} without proofs,
thus making our exposition self-contained.
In Section \ref{sec:main},
our main results are stated and proved.

\section{A subspace lattice and its hyper-cubic structure}\label{hyper-cubic}

We now begin our formal argument.
Recall  the integers 
$\mathbb{Z} = \lbrace 0, \pm 1, \pm 2, \ldots \rbrace$
and the natural numbers
$\mathbb{N} = \lbrace 0, 1, 2, \ldots \rbrace$
and let $\mathbb{C}$ denote the complex field.
The Kronecker delta is denoted by $\delta$.
Throughout the paper except Section \ref{Uq(sl2hat)},
we fix $N \in \mathbb{N} \setminus \lbrace 0 \rbrace$.
Throughout the paper except 
Sections \ref{Ferrers}, \ref{sec:kappa} and \ref{Uq(sl2hat)},
we fix a prime power $q$.
Let $\mathbb{F}_q$ denote a finite field of $q$ elements
and let $H$ denote a vector space over
$\mathbb{F}_q$ with dimension $N$.
Let $P$ denote the set of all subspaces of $H$.
We view $P$ as a poset
with the partial order given by inclusion.
The poset $P$ is a graded lattice of rank $N$ 
where the rank function is defined by its dimension and called the \emph{subspace lattice}.
For two subspaces $y,z \in P$, we say $y$ \emph{covers} $z$
whenever $z \subseteq y$ and $\dim z = \dim y - 1$.
By a \emph{(full) flag} on $H$ we mean a sequence $\lbrace x_i \rbrace_{i=0}^N$
of subspaces in $P$ such that $\dim x_i = i$ for $0 \le i \le N$
and $x_{i-1} \subsetneq x_i$ for $1 \le i \le N$.
For the rest of this paper, 
we fix a flag $\lbrace x_i \rbrace_{i=0}^N$ on $H$.
A basis $v_1, v_2, \ldots, v_N$ for $H$
is said to be \emph{adapted to the flag $\lbrace x_i \rbrace_{i=0}^N$}
whenever each $x_i$ is spanned by $v_1, v_2, \ldots, v_i$ for $1 \le i \le N$.

By the \emph{$N$-cube} we mean the poset
consisting of all $N$-tuples in $\lbrace 0,1 \rbrace^N$
with the partial order $\mu \le \nu$ defined by
$\mu_m \le \nu_m$ for all $1 \le m \le N$,
where $\mu = (\mu_1,\mu_2,\ldots,\mu_N),
\nu = (\nu_1,\nu_2,\ldots,\nu_N) \in \lbrace 0,1 \rbrace^N$.
(We note that it is isomorphic to the Boolean lattice of all subsets of an $N$-set.)
The $N$-cube is a graded lattice of rank $N$
with the rank function defined by 
\[
|\mu| = \mu_1 + \mu_2 + \cdots + \mu_N
\]
for $\mu = (\mu_1,\mu_2,\ldots,\mu_N) \in \lbrace 0,1 \rbrace^N$.

\begin{prop}\label{location}
There exists an order-preserving map
from the subspace lattice $P$ to the $N$-cube which sends
$y \in P$ to $(\mu_1,\mu_2,\ldots,\mu_N) \in \lbrace 0,1 \rbrace^N$
where
\[
\dim (y \cap x_m) = \mu_1 + \mu_2 + \cdots + \mu_m
\]
for $1 \le m \le N$.
Moreover this map is surjective.
\end{prop}

\begin{proof}
Let $y \in P$ and $1 \le m \le N$.
We have $\dim(y \cap x_{m-1}) \le \dim(y \cap x_m)$
since $x_{m-1} \subseteq x_m$.
We also have $\dim(y \cap x_m) - \dim(y \cap x_{m-1})\le 1$ since 
$\dim x_m - \dim x_{m-1} = 1$.
Thus
$\mu_m = \dim(y \cap x_m) - \dim(y \cap x_{m-1})$
is either $0$ or $1$.
Therefore 
this correspondence becomes 
a map from $P$ to the $N$-cube.
Let $y, z \in P$ satisfy $y \subseteq z$ and
let $\mu = (\mu_1, \mu_2, \ldots, \mu_N), \lambda = (\lambda_1, \lambda_2, \ldots, \lambda_N) \in \lbrace 0,1 \rbrace^N$
be the images of $y, z$ under the map, respectively.
If there exists $1 \le m \le N$ such that  $z \cap x_{m-1} = z \cap x_m$,
then
\[
(y \cap x_m) \setminus (y \cap x_{m-1}) =
(y \cap x_m) \setminus x_{m-1} \subseteq
(z \cap x_m) \setminus x_{m-1} =
(z \cap x_m) \setminus (z \cap x_{m-1}) = \emptyset,
\]
and so
we have
$y \cap x_{m-1} = y \cap x_m$.
Therefore, $\lambda_m = 0$ implies $\mu_m = 0$ for any $1 \le m \le N$,
which is equivalent to $\mu \le \lambda$.
We have now proved that the map preserves the ordering.
To show its surjectivity,
let $v_1, v_2, \ldots, v_N$ denote a basis for $H$
adapted to the flag $\lbrace x_i \rbrace_{i=0}^N$.
For any $\mu = (\mu_1, \mu_2, \ldots, \mu_N) \in \lbrace 0,1 \rbrace^N$,
consider the subspace $y \in P$ spanned by
the vectors $\lbrace v_i \mid 1 \le i \le N, \mu_i = 1 \rbrace$.
For each $1 \le m \le N$,
the intersection $y \cap x_m$
is spanned by the vectors 
$\lbrace v_i \mid 1 \le i \le m, \mu_i = 1 \rbrace$.
Therefore,
$\dim (y \cap x_m) - \dim (y \cap x_{m-1}) = \mu_m$
for $1 \le m \le N$
and so $y$ is mapped to $\mu$.
This proves the map is surjective.
\end{proof}

\begin{defi}\label{defi:location}\normalfont
If $\mu \in \lbrace 0,1 \rbrace^N$ is the image of $y \in P$
by the map in Proposition \ref{location},
we call $\mu$ the \emph{location} of $y$.
For $\mu \in \lbrace 0,1 \rbrace^N$,
let $P_\mu$ denote the set of all subspaces at location $\mu$.
For notational convenience, for $\mu \in \mathbb{Z}^N$ we set $P_\mu = \emptyset$
unless $\mu \in \lbrace 0,1 \rbrace^N$.
\end{defi}

Note that 
$P$ is the disjoint union of $P_\mu$, where $\mu \in \lbrace 0,1 \rbrace^N$.
Observe that $\dim y = |\mu|$ for $y \in P_\mu$.

\begin{defi}\label{mcovers}\normalfont
Let $1 \le m \le N$.
For $\mu = (\mu_1, \mu_2, \ldots, \mu_N), \nu = (\nu_1, \nu_2, \ldots, \nu_N) \in \lbrace 0,1 \rbrace^N$,
we say $\mu$ \emph{$m$-covers} $\nu$ whenever 
$\nu_m < \mu_m$ and $\nu_n = \mu_n$ for $1 \le n \le N$ with $n \neq m$.
Similarly, for $y, z \in P$, we say 
$y$ \emph{$m$-covers} $z$
whenever $y$ covers $z$ and
the location of $y$ $m$-covers the location of $z$.
\end{defi}

For each $1 \le m \le N$,
let $\widehat{m}$ denote the $N$-tuple in $\lbrace 0,1 \rbrace^N$
with a $1$ in $m$-th coordinate and $0$ elsewhere.
To simplify the notation,
we consider the coordinate-wise addition in $\mathbb{Z}^N$ so that
$\mu$ $m$-covers $\nu$ if and only if $\mu = \nu + \widehat{m}$ for $\mu, \nu \in \lbrace 0, 1 \rbrace^N$.

\begin{lem}\label{lem:adj}
For $\mu = (\mu_1, \mu_2, \ldots, \mu_N) \in \lbrace 0,1 \rbrace^N$ and for $1 \le m \le N$,
the following (i), (ii) hold.
\begin{enumerate}
\item Given $y \in P_\mu$, the number of subspaces $m$-covered by $y$ is
\[
\delta_{\mu_m,1}q^{\mu_{m+1}+\mu_{m+2}+ \cdots + \mu_N}.
\]
\item Given $y \in P_\mu$, the number of subspaces which $m$-cover $y$ is
\[
\delta_{\mu_m,0}q^{(m-1) - (\mu_1 + \mu_2 + \cdots + \mu_{m-1})}.
\]
\end{enumerate}
\end{lem}

\begin{proof}
(i)
Let $\widetilde{P}$ be the set of subspaces in $P$ which are $m$-covered by $y$.
Then $\widetilde{P} \subseteq P_{\mu-\widehat{m}}$.
If $\mu_m = 0$, then $\mu-\widehat{m} \not\in \lbrace0,1\rbrace^N$ and so $\widetilde{P} = \emptyset$.
We may assume $\mu_m = 1$.
For $z \in \widetilde{P}$,
we have $y \cap x_{m-1} = z \cap x_{m-1} = z \cap x_m$
since $z \in P_{\mu-\widehat{m}}$, $y \in P_\mu$ and 
$z \subseteq y$.

Set $n = \dim y - \dim (y \cap x_m) = \mu_{m+1}+\mu_{m+2}+ \cdots + \mu_N$.
Let $U_n$ denote the set of $n$-sets of linearly independent vectors $\mathbf{u} = \lbrace u_1, u_2, \ldots, u_n \rbrace \subseteq y \setminus (y \cap x_m)$
such that $(\Span \mathbf{u}) \cap (y \cap x_m) = 0$.
Since $\dim y = |\mu|$ and $\dim (y \cap x_m) = \mu_1+\cdots+\mu_m$, we have
\[
|U_n|
= \prod_{k = 1}^{n}
(q^{|\mu|} - q^{\mu_1+\cdots+\mu_m+k-1}).
\]
For $\mathbf{u} \in U_n$ and $1 \le k \le m$,
we have $(\Span \mathbf{u}) \cap x_k = 0$.
For $\mathbf{u} \in U_n$ and $m \le k \le N$,
since $y = \Span \mathbf{u} + (y \cap x_m)$ and $y \cap x_m \subseteq x_k$,
we have
$y \cap x_k = \left(\Span \mathbf{u} + (y \cap x_m)\right) \cap x_k = (\Span \mathbf{u}) \cap x_k + (y \cap x_m)$.

We count the cardinality of the following set $S$ in two ways.
\[
S = \lbrace (\mathbf{u}, z)
\mid \mathbf{u} \in U_n, z \in \widetilde{P},
z = \Span \mathbf{u} + (y \cap x_{m-1})
\rbrace.
\]
Let $\mathbf{u} \in U_n$, and
set $z = \Span \mathbf{u} + (y \cap x_{m-1})$.
Since $\Span \mathbf{u} \subseteq y$,
we have $z \subseteq y$.
For $1 \le k \le m-1$, we have
$z \cap x_k = \left(\Span \mathbf{u} + (y \cap x_{m-1})\right) \cap x_k \supseteq y\cap x_k$,
and moreover, equality must hold since $z \subseteq y$.
So, we have $\dim (z \cap x_k) = \dim (y \cap x_k)$.

For $m \le k \le N$,
since $y \cap x_{m-1} \subseteq x_k$, we have
$z \cap x_k = \left(\Span \mathbf{u} + (y \cap x_{m-1})\right) \cap x_k = (\Span \mathbf{u} \cap x_k) + (y \cap x_{m-1})$.
Recall that 
$y \cap x_k = (\Span \mathbf{u} \cap x_k) + (y \cap x_m)$.
Since the sums in these two equations are direct,
we have
$\dim (z \cap x_k) - \dim (y \cap x_k) =  
\dim(y \cap x_m) - \dim(y \cap x_{m-1})
= 1$.
Thus, $z \in \widetilde{P}$.
By these comments, we have
\begin{equation}
|S| = |U_n|.
\label{eq:S1}
\end{equation}
Conversely, let $z \in \widetilde{P}$.
We have $\dim z - \dim (z \cap x_m) = n$ since
$z \in P_{\mu - \widehat{m}}$.
So, there exists an
$n$-set of linearly independent vectors $\mathbf{u} = \lbrace u_1, u_2, \ldots, u_n \rbrace \subseteq z \setminus (z \cap x_m)$
such that $z = \Span \mathbf{u} + (z \cap x_m)$
and $(\Span \mathbf{u}) \cap (z \cap x_m) = 0$.
Let $U_n(z)$ denote the set of such $n$-sets.
Since $\dim z = |\mu|-1$ and $\dim (z \cap x_m) = \mu_1+\cdots+\mu_m-1$,
we have
\[
|U_n(z)|
= \prod_{k = 1}^{n}
(q^{|\mu|-1} - q^{\mu_1+\cdots+\mu_m+k-2})
= q^{-n}|U_n|.
\]
For $\mathbf{u} \in U_n(z)$, we have
$\mathbf{u} \subseteq y \setminus (y \cap x_m)$
since $z \subseteq y$, and
we also have $(\Span \mathbf{u}) \cap (y \cap x_m) = (\Span \mathbf{u}) \cap (z \cap x_m) =0$
since $\Span \mathbf{u} \subseteq z \subseteq y$.
Thus, $U_n(z) \subseteq U_n$.
We may write $z = \Span \mathbf{u} + (y \cap x_{m-1})$
since $y \cap x_{m-1} = z \cap x_m$.
Moreover,
if $\mathbf{u} \in U_n$ satisfies $z = \Span \mathbf{u} + (y \cap x_{m-1})$,
then 
$\mathbf{u} \subseteq z$ and $(\Span \mathbf{u}) \cap (z \cap x_m) =0$,
which imply $\mathbf{u} \in U_n(z)$.
By these comments,  we have
$U_n(z) = \lbrace \mathbf{u} \in U_n \mid
z = \Span \mathbf{u} + (y \cap x_{m-1}) \rbrace$
for $z \in \widetilde{P}$, and so
\begin{equation}
|S|
= \sum_{z \in \widetilde{P}} |U_n(z)|
= |\widetilde{P}| \times q^{-n}|U_n|.
\label{eq:S2}
\end{equation}
Thus, by \eqref{eq:S1} and \eqref{eq:S2},
we have $|\widetilde{P}| = q^n$.
The result follows.

(ii)
Let $\widetilde{P}$ be the set of subspaces in $P$ which $m$-cover $y$.
Then $\widetilde{P} \subseteq P_{\mu+\widehat{m}}$.
If $\mu_m = 1$, then $\mu+\widehat{m} \not\in \lbrace0,1\rbrace^N$ and so $\widetilde{P} = \emptyset$.
We may assume $\mu_m = 0$.
Let $U = x_m \setminus x_{m-1}$.
We have $|U| = q^m - q^{m-1}$.
We also have
$U \cap y \subseteq (y \cap x_m) \setminus (y \cap x_{m-1}) = \emptyset$.

We count the cardinality of the following set $S$ in two ways.
\[
S = \lbrace (u, z)
\mid u \in U, z \in \widetilde{P},
z = (\Span u) + y
\rbrace.
\]
Let $u \in U$, and set $z = (\Span u) + y$.
Then $y \subseteq z$.
For $m \le k \le N$,
since $\Span u \subseteq x_k$, we have
$z \cap x_k = (\Span u + y) \cap x_k = (\Span u) + (y \cap x_k)$.
Since the sum is direct, we have 
$\dim(z \cap x_k) = \dim(y \cap x_k) + 1$.
For $1 \le k \le m-1$,
since $y \cap x_m \subseteq x_{m-1}$
and $\Span u \cap x_{m-1} = 0$, we have
\begin{align*}
z \cap x_k &= (z \cap x_m) \cap x_{m-1} \cap x_k \\
&= (\Span u + y \cap x_m) \cap x_{m-1} \cap x_k \\
&= (\Span u \cap x_{m-1} + y \cap x_{m-1}) \cap x_k \\
&= y \cap x_k.
\end{align*}
In particular, $\dim (z \cap x_k) = \dim (y \cap x_k)$.
Thus $z \in \widetilde{P}$.
By these comments, we have
\begin{equation}
|S| = |U|.
\label{eq:S3}
\end{equation}
Conversely, let $z \in \widetilde{P}$.
We have $\dim (z \cap x_m) - \dim (z \cap x_{m-1}) = 1$
since $z \in P_{\mu+\widehat{m}}$.
Denote $U(z) = (z \cap x_m) \setminus (z \cap x_{m-1})$.
For any $u \in U(z)$,
we have $z \cap x_m = \Span u + (z \cap x_{m-1})$.
Since $\dim (z \cap x_m) = \mu_1 + \cdots + \mu_{m-1}+1$ and $\dim (z \cap x_{m-1}) = \mu_1 + \cdots + \mu_{m-1}$,
we have
\[
|U(z)| = q^{\mu_1 + \cdots + \mu_{m-1}+1} - q^{\mu_1 + \cdots + \mu_{m-1}} = q^{(\mu_1+\cdots+\mu_{m-1})-(m-1)}|U|.
\]
By definition, we have $U(z) \subseteq U$.
Let $u \in U$ satisfy $z = \Span u + y$.
Then $u \in z$, and so $u \in U(z)$.
By these comments, we have
$\lbrace u \in U \mid z = (\Span u) + y\rbrace = U(z)$, and so
\begin{equation}
|S| 
= \sum_{z \in \widetilde{P}} |U(z)|
= |\widetilde{P}| \times q^{(\mu_1+\cdots+\mu_{m-1})-(m-1)}|U|.
\label{eq:S4}
\end{equation}
Thus, by \eqref{eq:S3} and \eqref{eq:S4},
we have $|\widetilde{P}| = q^{(m-1)-(\mu_1+\cdots+\mu_{m-1})}$.
The result follows.
\end{proof}

\begin{lem}\label{lem:count}
Let $1 \le m < n \le N$.
For $\mu = (\mu_1, \mu_2, \ldots, \mu_N) \in \lbrace 0,1 \rbrace^N$ with $\mu_m = \mu_n = 1$, 
the following hold.
\begin{enumerate}
\item Given $z \in P_\mu$ and $y \in P_{\mu - \widehat{m} - \widehat{n}}$
with $y \subseteq z$,
there exists a unique element in $P_{\mu - \widehat{n}}$
which $m$-covers $y$ and which is $n$-covered by $z$.
\item Given $z \in P_\mu$ and $y \in P_{\mu - \widehat{m} - \widehat{n}}$
with $y \subseteq z$,
there exist exactly $q$ elements in $P_{\mu - \widehat{m}}$
which $n$-cover $y$ and which are $m$-covered by $z$.
\item Given $y \in P_{\mu - \widehat{m}}$ and $z \in P_{\mu - \widehat{n}}$,
if there exists an element that is covered by both $y$ and $z$,
then there exists a unique element that covers both $y$ and $z$.
\item Given $y \in P_{\mu - \widehat{m}}$ and $z \in P_{\mu - \widehat{n}}$,
if there exists an element that covers both $y$ and $z$,
then there exists a unique element that is covered by both $y$ and $z$.
\end{enumerate}
\end{lem}

\begin{proof}
(i)
We first show the existence of such element.
Set $w = y + (z \cap x_{n-1})$
and let $\mu'$ denote the location of $w$.
We have $y \subseteq w \subseteq z$ and $\mu - \widehat{m} - \widehat{n} \le \mu' \le \mu$.
Since $m \le n-1$, we have
$w \cap x_m = (y + z \cap x_{n-1}) \cap x_m \supseteq z \cap x_m$,
and moreover equality must hold since $w \subseteq z$.
So, $\dim (w \cap x_m) = \dim (z \cap x_m)$.
Since $z \cap x_{n-1} \subseteq x_n$, we have
$w \cap x_n = (y + z \cap x_{n-1}) \cap x_n = y \cap x_n + z \cap x_{n-1}$.
Since $z \in P_\mu$ and $y \in P_{\mu - \widehat{m} - \widehat{n}}$,
we have $\dim (w \cap x_n) = \dim (y \cap x_n) + \dim (z \cap x_{n-1}) - \dim (y \cap x_{n-1}) = \dim (y \cap x_n) - 1$.
Thus the location $\mu'$ must be $\mu - \widehat{n}$,
i.e., $w \in P_{\mu - \widehat{n}}$.
Since $y \subseteq w \subseteq z$,
the element $w$ must $m$-cover $y$ and be $n$-covered by $z$.

We next show the uniqueness of such element.
Take any $w' \in P_{\mu - \widehat{n}}$
which covers $y$ and which is covered by $z$.
Then $w'$ must contain both $y$ and $z \cap x_{n-1}$.
So $w \subseteq w'$. By computing dimensions, $w$ and $w'$ must coincide.
The result follows.

(ii)
Let $\widetilde{P}$ be the set of subspaces in $P$ which cover $y$ and which are covered by $z$.
Since $\dim (z/y) = 2$,
we have $|\widetilde{P}| = (q^2-1)/(q-1) = q+1$.
Let $w \in \widetilde{P}$ and let $\mu'$ be the location of $w$.
Then, we have $\mu' \in \lbrace \mu - \widehat{n}, \mu - \widehat{m}\rbrace$.
So, the $q+1$ elements in $\widetilde{P}$
must belong to either $P_{\mu - \widehat{n}}$ or $P_{\mu - \widehat{m}}$.
Therefore the result follows from (i).

(iii)
Let $w$ be an element that is covered by both $y$ and $z$.
Then we have $w \subseteq y \cap z$.
Since $y$ and $z$ are distinct, we have
$\dim y - 1 = \dim w \le \dim (y \cap z) \le \dim y - 1$,
and so $w = y \cap z$.
Set $w' = y+z$.
Then $\dim w' = \dim y + 1 = \dim z + 1$.
This means $w'$ is an element 
that covers both $y$ and $z$.
The uniqueness is clear.

(iv) Similar to (iii).
\end{proof}

\section{Ferrers boards}\label{Ferrers}

We introduce the notion of Ferrers boards.
For the general theory on this topic, we refer the reader to \cite[Chapters 1 and 2]{St}.
Note that we modify the notations of \cite{St} to fit our setting.

Let $\mu = (\mu_1, \mu_2, \ldots, \mu_N) \in \lbrace 0,1 \rbrace^N$.
Then $\mu$ has a natural correspondence 
with a bipartition of $\lbrace 1, 2, \ldots, N\rbrace$,
which is defined by
\begin{align}
S_\mu = \lbrace s \in \mathbb{N} \mid 1 \le s \le N, \mu_s = 0 \rbrace,
&&
T_\mu = \lbrace t \in \mathbb{N} \mid 1 \le t \le N, \mu_t = 1 \rbrace.
\label{STmu}
\end{align}
We remark that
$S_\mu$ and $T_\mu$ are empty if and only if
$\mu = \mathbf{1} = (1,1,\ldots,1)$ and $\mu = \mathbf{0} = (0,0,\ldots,0)$, respectively.
The \emph{Ferrers board} of shape $\mu$ is defined by
\begin{equation}
B_\mu = \lbrace
(s,t) \in S_\mu \times T_\mu \mid s < t
\rbrace.
\label{Bmu}
\end{equation}
If both $S_\mu$ and $T_\mu$ are not empty,
i.e.\ if $\mu \neq \mathbf{0}, \mathbf{1}$,
we can draw a Ferrers board as  
a two-dimensional subarray of a matrix whose rows indexed by $S_\mu$
and columns indexed by $T_\mu$,
whose $(s,t)$-entry has a box for all $(s,t) \in B_\mu$.

This subarray is also known as a \emph{Young diagram} of shape $\mu$.

\begin{ex}[$N=13$]\label{ex1}
Let $\mu = (0,1,1,0,1,1,0,1,1,0,0,1,0) \in \lbrace 0,1 \rbrace^{13}$.
Then the corresponding Ferrers board $B_\mu$ has the following subarray form:
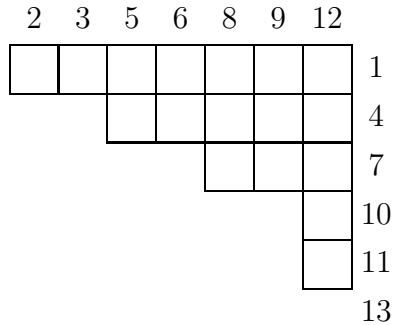
\[
\begin{ytableau}
   \none[2]&\none[3]&\none[5]&\none[6]&\none[8]&\none[9]&\none[12]&\none[] \\
   *(white)&*(white)&*(white)&*(white)&*(white)&*(white)&*(white)&\none[1] \\
   \none[]&\none[]&*(white)&*(white)&*(white)&*(white)&*(white)&\none[4] \\
   \none[]&\none[]&\none[]&\none[]&*(white)&*(white)&*(white)&\none[7] \\
   \none[]&\none[]&\none[]&\none[]&\none[]&\none[]&*(white)&\none[10] \\
   \none[]&\none[]&\none[]&\none[]&\none[]&\none[]&*(white)&\none[11] \\
   \none[]&\none[]&\none[]&\none[]&\none[]&\none[]&\none[]&\none[13] \\
\end{ytableau}
\]
\end{ex}

Take a nonempty Ferrers board $B_\mu$ of shape $\mu$.
For $(s_0,t_0) \in B_\mu$,
the \emph{rectangle} in $B_\mu$ with respect to $(s_0,t_0)$,
denoted by $B_\mu(s_0,t_0)$, is defined by
\begin{equation}
B_\mu(s_0,t_0) =
\lbrace
(s, t) \in B_\mu \mid s \le s_0, t \ge t_0
\rbrace.
\label{Bmu(s,t)}
\end{equation}
It is actually the rectangle in the corresponding Young diagram
which includes the top-right corner 
and the $(s_0,t_0)$-th box as its bottom-left corner.
We remark that such a rectangle
is called the \emph{Durfee square}
if it is the largest square in $B_\mu$.
To see the rectangle structure, we use the following notation:
\begin{align}
S_\mu(m) = \lbrace s \in S_\mu \mid  s \le m \rbrace,
&&
T_\mu(m) = \lbrace t \in T_\mu \mid t \ge m \rbrace,
\label{STmu(s,t)}
\end{align}
for $1 \le m \le N$
so that we can write $B_\mu(s_0,t_0) = S_\mu(s_0) \times T_\mu(t_0)$.

\begin{ex}[$N=13$]\label{ex2}
Take $\mu \in \lbrace 0,1 \rbrace^{13}$ as in Example \ref{ex1}.
Then $(4,6) \in B_\mu$ and
the rectangle $B_\mu(4,6)$ is the 
set of the following eight elements:
\begin{align*}
(1,6), && (1,8), && (1,9), && (1,12), &&
(4,6), && (4,8), && (4,9), && (4,12).
\end{align*}
In the corresponding Young diagram,
$B_\mu(4,6)$ is the following gray rectangle.
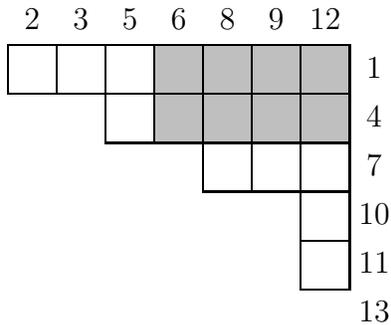
\[
\begin{ytableau}
   \none[2]&\none[3]&\none[5]&\none[6]&\none[8]&\none[9]&\none[12]&\none[] \\
   *(white)&*(white)&*(white)&*(gray!50)&*(gray!50)&*(gray!50)&*(gray!50)&\none[1] \\
   \none[]&\none[]&*(white)&*(gray!50)&*(gray!50)&*(gray!50)&*(gray!50)&\none[4] \\
   \none[]&\none[]&\none[]&\none[]&*(white)&*(white)&*(white)&\none[7] \\
   \none[]&\none[]&\none[]&\none[]&\none[]&\none[]&*(white)&\none[10] \\
   \none[]&\none[]&\none[]&\none[]&\none[]&\none[]&*(white)&\none[11] \\
   \none[]&\none[]&\none[]&\none[]&\none[]&\none[]&\none[]&\none[13] \\
\end{ytableau}
\]
\end{ex}

Take a nonempty
Ferrers board $B_\mu$ of shape $\mu$.
A subset of $B_\mu$
such that no two elements have a common entry
is called a \emph{rook placement} on $B_\mu$.
Let $\sigma$ denote a rook placement on $B_\mu$.
The \emph{row index set} $\pi_1(\sigma)$ and 
the \emph{column index set} $\pi_2(\sigma)$ of $\sigma$ are defined by
\begin{align}
\pi_1(\sigma) = \lbrace s \in S_\mu \mid \text{$(s,t) \in \sigma$ for some $t$}\rbrace,
&&
\pi_2(\sigma) = \lbrace t \in T_\mu \mid \text{$(s,t) \in \sigma$ for some $s$}\rbrace,
\label{pisigma}
\end{align}
respectively.
Remark that $|\pi_1(\sigma)| = |\pi_2(\sigma)| = |\sigma|$.
Assume $\sigma \neq \emptyset$.
For $1 \le i \le |\sigma|$,
we denote by $s_i$ and by $t_i$ the $i$-th smallest element in $\pi_1(\sigma)$ and in $\pi_2(\sigma)$, respectively.
Then  $\sigma$ gives rise to a permutation of
$\lbrace 1, 2, \ldots, |\sigma|\rbrace$
which sends $i$ to $j$ where $(s_i, t_j) \in \sigma$.

\begin{lem}\label{PIcond}
Let $\mu \in \lbrace 0,1 \rbrace^N$ and $\sigma$ be a rook placement on $B_\mu$
with the row/column index sets $\pi_1 = \pi_1(\sigma)$, $\pi_2 = \pi_2(\sigma)$, respectively.
Then the pair $(\pi_1, \pi_2)$ satisfies the following.
\begin{enumerate}
\item $|\pi_1| = |\pi_2|$.
\item 
Let $n$ denote the common value in (i).
For $1 \le i \le n$,
the $i$-th smallest element in $\pi_1$ is strictly smaller than
the $i$-th smallest element in $\pi_2$.
\end{enumerate}
\end{lem}

\begin{proof}
(i)
It is clear.

(ii)
We may assume $\sigma \neq \emptyset$ since otherwise the assertion is clear.
Let $\widetilde{\sigma}$ denote the permutation of $\lbrace 1, 2, \ldots, n \rbrace$
corresponding to $\sigma$.
For $1 \le i \le n$, we write $s_i$, $t_i$ for the $i$-th smallest element in $\pi_1$, $\pi_2$, respectively.
Fix $1 \le i \le n$.
Since $\widetilde{\sigma}$ is a permutation,
there exists $i \le k \le n$ such that $\widetilde{\sigma}(k) \le i$.
So we have $(s_k, t_{\widetilde{\sigma}(k)}) \in \sigma$.
Therefore $s_i \le s_k < t_{\widetilde{\sigma}(k)} \le t_i$
as desired.
\end{proof}

\begin{prop}\label{prop:condsigma}
Let $\mu \in \lbrace 0,1 \rbrace^N$.
For a pair $(\pi_1, \pi_2)$ such that $\pi_1 \subseteq S_\mu$ and $\pi_2 \subseteq T_\mu$,
the following are equivalent:
\begin{enumerate}
\item there exists a rook placement $\sigma$ on $B_\mu$ such that $\pi_1 = \pi_1(\sigma)$
and $\pi_2 = \pi_2(\sigma)$;
\item it satisfies (i), (ii) in Lemma \ref{PIcond}.
\end{enumerate}
\end{prop}

\begin{proof}
We have shown in Lemma \ref{PIcond}
that (i) implies (ii).

Suppose we are given $\pi_1 \subseteq S_\mu$ and $\pi_2 \subseteq T_\mu$
satisfying (i), (ii) in Lemma \ref{PIcond}.
By the condition (i) in Lemma \ref{PIcond}, we set $n = |\pi_1| = |\pi_2|$.
Let $\sigma = \lbrace (s_i, t_i) \mid 1 \le i \le n\rbrace$, where
each $s_i$, $t_i$ is the $i$-th smallest element in $\pi_1$, $\pi_2$, respectively.
By the condition (ii) in Lemma \ref{PIcond}, we have $\sigma \subseteq B_\mu$
and so $\sigma$ is a rook placement on $B_\mu$.
By construction, it is clear that $\pi_1 = \pi_1(\sigma)$ and $\pi_2 = \pi_2(\sigma)$.
So (ii) implies (i).
\end{proof}

\begin{defi}\label{sigmatype}\normalfont
Let $\mu \in \lbrace 0,1 \rbrace^N$
and consider the Ferrers board $B_\mu$ of shape $\mu$.
Then the \emph{type} of a rook placement $\sigma$ on $B_\mu$
is defined by
the disjoint union
\[
\pi_1(\sigma) \cup \pi_2(\sigma) \subseteq \lbrace 1, 2, \ldots, N \rbrace,
\]
where $\pi_1(\sigma)$, $\pi_2(\sigma)$ are the row/column index sets of $\sigma$ defined in \eqref{pisigma}.
\end{defi}

\begin{lem}\label{LSmuLTmu}
Let $\mu \in \lbrace 0,1 \rbrace^N$.
For $\lambda \subseteq \lbrace 1, 2, \ldots, N \rbrace$,
the following are equivalent:
\begin{enumerate}
\item there exists a rook placement on $B_\mu$ of type $\lambda$;
\item the pair $(\lambda \cap S_\mu, \lambda \cap T_\mu)$ satisfies (i), (ii) in Lemma \ref{PIcond}.
\end{enumerate}
\end{lem}

\begin{proof}
Immediate from Proposition \ref{prop:condsigma}.
\end{proof}

\begin{lem}\label{cardeven}
For $\lambda \subseteq \lbrace 1, 2, \ldots, N \rbrace$,
the following are equivalent:
\begin{enumerate}
\item there exists a rook placement on $B_\mu$ of type $\lambda$ for some $\mu \in \lbrace 0,1 \rbrace^N$;
\item the cardinality of $\lambda$ is even. 
\end{enumerate}
\end{lem}

\begin{proof}
Fix $\lambda \subseteq \lbrace 1, 2, \ldots, N \rbrace$.
Suppose 
there exists a rook placement $\sigma$
on $B_\mu$ of type $\lambda$ for some $\mu \in \lbrace 0,1 \rbrace^N$.
Then by Lemma \ref{LSmuLTmu},
the pair $(\lambda \cap S_\mu, \lambda \cap T_\mu)$ satisfies (i), (ii) in Lemma \ref{PIcond}.
In particular, $|\lambda| = |\lambda \cap S_\mu| + |\lambda \cap T_\mu|$ is even.
So (ii) holds.

Conversely,
we suppose $|\lambda| = 2n$ for some $n \in \mathbb{N}$ and show (i) holds.
Let $(\pi_1, \pi_2)$ denote the bipartition of $\lambda$ where $\pi_1$ contains the first $n$ smallest elements in $\lambda$
and $\pi_2$ contains the remaining $n$ elements in $\lambda$.
Take any $\mu \in \lbrace 0,1 \rbrace^N$ such that $\pi_1 \subseteq S_\mu$ and $\pi_2 \subseteq T_\mu$.
Then  we have $\pi_1 = \lambda \cap S_\mu$ and $\pi_2 = \lambda \cap T_\mu$.
Observe that the pair $(\pi_1, \pi_2)$ satisfies (i), (ii) in Lemma \ref{PIcond}.
So by Lemma \ref{LSmuLTmu},
there exists a rook placement on $B_\mu$ of type $\lambda$.
In particular, (i) holds.
\end{proof}

Since rook placements
can be seen as permutations,
we define the concept of inversions.
Let $\sigma$ be a nonempty rook placement on a Ferrers board $B_\mu$ of shape $\mu$.
For $(s_0,t_0) \in \sigma$,
the \emph{local inversion number of $\sigma$ at $(s_0,t_0)$},
denoted by $\mathrm{inv}(\sigma, s_0, t_0)$,
is defined by
\begin{equation}
\mathrm{inv}(\sigma, s_0, t_0) =
|\lbrace (s, t) \in \sigma \mid s < s_0, t > t_0 \rbrace|= |\sigma \cap B_\mu(s_0,t_0)| - 1.
\label{LIN}
\end{equation}
For a rook placement $\sigma$,
the \emph{(total) inversion number of $\sigma$}, denoted by $\mathrm{inv}(\sigma)$,
is defined by
\[
\mathrm{inv}(\sigma) = \sum_{(s,t) \in \sigma} \mathrm{inv}(\sigma, s, t).
\]

\begin{ex}[$N=13$]\label{ex3}
Take $\mu \in \lbrace 0,1 \rbrace^{13}$ as in Example \ref{ex1}.
Consider the following rook placement $\sigma$ on $B_\mu$:
\[
\sigma = \lbrace (1,9), (4,6), (10,12)\rbrace.
\]
Then we have 
$\mathrm{inv}(\sigma, 1, 9) = \mathrm{inv}(\sigma, 10, 12) = 0$
and $\mathrm{inv}(\sigma, 4,6) = 1$.
Thus $\mathrm{inv}(\sigma) = 1$.
\[
\begin{ytableau}
   \none[2]&\none[3]&\none[5]&\none[6]&\none[8]&\none[9]&\none[12]&\none[] \\
   *(white)&*(white)&*(white)&*(white)&*(white)&*(white)\star&*(white)&\none[1] \\
   \none[]&\none[]&*(white)&*(white)\star&*(white)&*(white)&*(white)&\none[4] \\
   \none[]&\none[]&\none[]&\none[]&*(white)&*(white)&*(white)&\none[7] \\
   \none[]&\none[]&\none[]&\none[]&\none[]&\none[]&*(white)\star&\none[10] \\
   \none[]&\none[]&\none[]&\none[]&\none[]&\none[]&*(white)&\none[11] \\
   \none[]&\none[]&\none[]&\none[]&\none[]&\none[]&\none[]&\none[13] \\
\end{ytableau}
\]
\end{ex}

\begin{lem}\label{lem3.9}
Let $\mu \in \lbrace 0,1 \rbrace^N$
and
let
$\lambda \subseteq \lbrace 1, 2, \ldots, N \rbrace$ satisfy (ii) in Lemma \ref{LSmuLTmu}.
For $1 \le m \le N$ and for a rook placement $\sigma$ on $B_\mu$ of type $\lambda$,
we have
\begin{equation}
|\sigma \cap (S_\mu(m) \times T_\mu(m))|
=
|\lambda \cap S_\mu(m)| + |\lambda \cap T_\mu(m)| - \frac{|\lambda|}{2}. 
\label{eq:sigmamm}
\end{equation}
In particular, this number is independent of the choice of $\sigma$.
\end{lem}

\begin{proof}
Since $\pi_1(\sigma) = \lambda \cap S_\mu$, $\pi_2(\sigma) = \lambda \cap T_\mu$ and
$\sigma$ is a rook placement,
we have
\begin{align*}
|\lambda \cap S_\mu(m)| &= 
|\pi_1(\sigma) \cap S_\mu(m)| = 
|\sigma \cap (S_\mu(m) \times T_\mu)|, \\
|\lambda \cap T_\mu(m)| &= 
|\pi_2(\sigma) \cap T_\mu(m)| = 
|\sigma \cap (S_\mu \times T_\mu(m))|,\\
|\lambda| &= |\pi_1(\sigma)| + |\pi_2(\sigma)| = 2|\sigma|.
\end{align*}
Set $\overline{S_\mu(m)} = S_\mu \setminus S_\mu(m)$ and $\overline{T_\mu(m)} = T_\mu \setminus T_\mu(m)$.
Then we have
\begin{align*}
&|\lambda \cap S_\mu(m)| + |\lambda \cap T_\mu(m)| - \frac{|\lambda|}{2} \\
&= |\sigma \cap (S_\mu(m) \times T_\mu)| + |\sigma \cap (S_\mu \times T_\mu(m))| - |\sigma| \\
&= |\sigma \cap (S_\mu(m) \times T_\mu(m))| - |\sigma \cap (\overline{S_\mu(m)} \times \overline{T_\mu(m)})|.
\end{align*}
So, it remains to show that 
$\sigma \cap (\overline{S_\mu(m)} \times \overline{T_\mu(m)}) = \emptyset$.
Observe that 
\[
B_\mu \cap \left(\overline{S_\mu(m)} \times \overline{T_\mu(m)}\right) = 
\lbrace (s,t) \in B_\mu \mid t < m < s \rbrace = \emptyset.
\]
Therefore,
from $\sigma \subseteq B_\mu$, the result follows.
\end{proof}

The next lemma is a generalization of 
\cite[Corollary 1.3.10]{St}
and the proof of the next lemma is motivated by that of \cite[Corollary 1.3.10]{St}.

\begin{lem}\label{qinvsigma}
Let $\mu \in \lbrace 0,1 \rbrace^N$ and let
$\lambda \subseteq \lbrace 1, 2, \ldots, N \rbrace$ satisfy (ii) in Lemma \ref{LSmuLTmu}.
For $1 \le m \le N$, 
let $\rho(m,\mu,\lambda)$ denote the left-hand side of \eqref{eq:sigmamm}.
Then for $q \in \mathbb{C}$ with $q \neq 0, 1$, we have
\[
\sum_\sigma q^{\mathrm{inv}(\sigma)} = 
\prod_{s \in \lambda \cap S_\mu} \frac{q^{\rho(s,\mu,\lambda)} - 1}{q-1},
\]
where the sum is taken over all 
rook placements $\sigma$ on $B_\mu$ of type $\lambda$.
\end{lem}

\begin{proof}
If $\lambda = \emptyset$, the assertion is clear.
(Note that $\mathrm{inv}(\emptyset) = 0$.)
We assume $\lambda \neq \emptyset$.
We claim that
there exists a bijection between the following two sets:
\begin{enumerate}
\item rook placements $\sigma$ on $B_\mu$ of type $\lambda$,
\item integer sequences $(a_s)_{s \in \lambda \cap S_\mu}$
such that $0 \le a_s \le \rho(s,\mu,\lambda) - 1$
for $s \in \lambda \cap S_\mu$,
\end{enumerate}
such that $\mathrm{inv}(\sigma) = \sum_{s \in \lambda \cap S_\mu}a_s$.
Suppose for the moment that the claim is true.
Then we have
\[
\sum_{\sigma}q^{\mathrm{inv}(\sigma)}
=
\prod_{s \in \lambda \cap S_\mu}\left( \sum_{a_s = 0}^{\rho(s,\mu,\lambda)-1} q^{a_s}\right)
=
\prod_{s \in \lambda \cap S_\mu} \frac{q^{\rho(s,\mu,\lambda)} - 1}{q-1}.
\]
So the result follows.

Therefore, it remains to prove the claim.
For a given rook placement $\sigma$ on $B_\mu$ of type $\lambda$
and for $s \in \lambda \cap S_\mu$,
there exists a unique $t(s) \in \lambda \cap T_\mu$ such that $(s,t(s)) \in \sigma$.
Thus, we consider the map $\iota$ that sends $\sigma$ to $(a_s)_{s \in \lambda \cap S_\mu}$, where $a_s = \mathrm{inv}(\sigma, s, t(s))$.
Then for $s \in \lambda \cap S_\mu$, we have
\begin{align*}
0 \le a_s &= |\sigma \cap (S_\mu(s) \times T_\mu(t(s)))| - 1 \\
&\le |\sigma \cap (S_\mu(s) \times T_\mu(s))| - 1 \\
&= \rho(s,\mu,\lambda) - 1,
\end{align*}
where the second inequality follows from the fact that $s \le t(s)$.
This implies that the map $\iota$ is from (i) to (ii).
To show the bijectivity of $\iota$,
take a sequence $(a_s)_{s \in \lambda \cap S_\mu}$ in the set (ii).
Set $r = |\lambda \cap S_\mu|$
and for $1 \le i \le r$,
we write $s_i$ the $i$-th smallest element in $\lambda \cap S_\mu$.
By definition,
observe that
\[
0 \le a_{s_i} \le \rho(s_i,\mu,\lambda) - 1 \le 
|\lambda \cap S_\mu(s_i)| - 1 = i - 1,
\]
where the third inequality follows from
$|\lambda|/2 - |\lambda \cap T_\mu(s_i)| = |\lbrace t \in \lambda \cap T_\mu \mid t < s_i \rbrace| \ge 0$.
Then,
there exists a unique permutation $\widetilde{\sigma}$ of $\lbrace 1, 2, \ldots, r\rbrace$
such that 
\begin{equation}
a_{s_i} = |\lbrace j \mid 1 \le j < i, \widetilde{\sigma}(i) < \widetilde{\sigma}(j) \rbrace|.
\label{ai}
\end{equation}
Then consider the set
$\sigma = \lbrace (s_i, t_{\widetilde{\sigma}(i)}) \mid 1 \le i \le r\rbrace$,
where $t_i$ is the $i$-th smallest element in $\lambda \cap T_\mu$.
Fix $1 \le i \le r$.
By \eqref{ai},
we have $\widetilde{\sigma}(i) \ge i - a_{s_i}$ and
so we have
\[
\widetilde{\sigma}(i) \ge i - a_{s_i} \ge i - \rho(s_i,\mu,\lambda) + 1 =
|\lbrace t \in \lambda \cap T_\mu \mid t < s_i \rbrace| + 1.
\]
This implies that $s_i < t_{\widetilde{\sigma}(i)}$.
This holds for any $1 \le i \le r$ and so
$\sigma$ becomes a rook placement on $B_\mu$.
It is clear that $\sigma$ is of type $\lambda$.
By construction, the map which sends $(a_s)_{s \in \lambda \cap S_\mu}$ to $\sigma$
becomes the inverse of $\iota$.
Therefore, our claim holds.
\end{proof}

\section{The matrix representation of $P$}\label{matrixrep}

For a field $\mathbb{K}$ and for two finite nonempty sets $S$ and $T$,
let $\mathrm{Mat}_{S,T}(\mathbb{K})$ denote
the set of all matrices 
with rows indexed by $S$ and columns indexed by $T$
whose entries are in $\mathbb{K}$.
If $S = T$, we write it $\mathrm{Mat}_{S}(\mathbb{K})$ for short.
For $M \in \mathrm{Mat}_{S,T}(\mathbb{K})$,
the \emph{support} of $M$,
denoted by $\mathrm{Supp}(M)$,
is the set of indices containing nonzero entries:
\[
\mathrm{Supp}(M) =
\lbrace (s,t) \in S \times T \mid M_{s,t} \neq 0 \rbrace.
\]

For $\mu \in \lbrace 0,1 \rbrace^N$,
recall the corresponding bipartition $S_\mu, T_\mu$ from \eqref{STmu}
and the Ferrers board $B_\mu$ of shape $\mu$ from \eqref{Bmu}.
We will assume $\mu \neq \mathbf{0}$, $\mathbf{1}$ 
in this section
so that both $S_\mu$ and $T_\mu$ are nonempty.

\begin{defi}\label{supportmu}\normalfont
Let $\mu \in \lbrace 0,1 \rbrace^N$ with $\mu \neq \mathbf{0}$, $\mathbf{1}$.
Let $\mathcal{M}_\mu(\mathbb{F}_q)$ denote the set of matrices in $\mathrm{Mat}_{S_\mu,T_\mu}(\mathbb{F}_q)$
such that $\mathrm{Supp}(M) \subseteq B_\mu$.
\end{defi}

Recall the set $P_\mu$ of subspaces at location $\mu \in \lbrace 0,1 \rbrace^N$ from Definition \ref{defi:location}.

\begin{prop}\label{yY}
Let $\mu \in \lbrace 0,1 \rbrace^N$ with $\mu \neq \mathbf{0}$, $\mathbf{1}$.
Fix a basis $v_1, v_2, \ldots, v_N$ for $H$ adapted to the flag $\lbrace x_i \rbrace_{i=0}^N$.
There exists a bijection from $P_\mu$
to the set 
$\mathcal{M}_\mu(\mathbb{F}_q)$ in Definition \ref{supportmu}
that sends $y \in P_\mu$ to $Y \in \mathcal{M}_\mu(\mathbb{F}_q)$,
where $y$ has a basis
\begin{align*}
\sum_{s \in S_\mu} Y_{s,t}v_s + v_t,
&&
t \in T_\mu.
\end{align*}
\end{prop}

\begin{proof}
For $y \in P_\mu$, there exists a basis $w_t$ $(t \in T_\mu)$
for $y$
such that $w_t \in x_t \setminus x_{t-1}$ for each $t \in T_\mu$.
Write each vector $w_t$ as a linear combination of the fixed basis $v_1, v_2, \ldots, v_t$ for $x_t$.
Without loss of generality,
we may assume 
the coefficient of $v_t$ is $1$.
Use linear operations on the basis $w_t$ ($t \in T_\mu$) to make the coefficient of $v_{t'}$
$0$ for any $t' \in T_\mu$ with $t \neq t'$.
Observe that the resulting basis $w_t'$ $(t \in T_\mu)$
is uniquely determined by $y$.
Then from the basis $w_t'$ ($t \in T_\mu$), we construct the matrix
$Y \in \mathrm{Mat}_{S_\mu,T_\mu}(\mathbb{F}_q)$ such that
$Y_{s,t}$ is the coefficient of $v_s$ in $w_t'$.
Then we have $Y \in \mathcal{M}_\mu(\mathbb{F}_q)$
since $w_t' \in x_t$.
On the other hand,
let $Y \in \mathcal{M}_\mu(\mathbb{F}_q)$.
For $t \in T_\mu$, we write $w_t = \sum_{s \in S_\mu} Y_{s,t}v_s + v_t$.
Since $\mathrm{Supp}(Y) \subseteq B_\mu$,
the vector $w_t$ is a linear combination of $v_1, v_2, \ldots, v_t$,
that means $w_t \in x_t \setminus x_{t-1}$.
Therefore the subspace $y$ spanned by the vectors $w_t$ ($t \in T_\mu$) must belong to $P_\mu$.
\end{proof}

\begin{defi}\label{matrixform}\normalfont
Let $\mu \in \lbrace 0,1 \rbrace^N$ with $\mu \neq \mathbf{0}$, $\mathbf{1}$.
Take $y \in P_\mu$.
By the \emph{matrix form} of $y$, we mean
the matrix $Y \in \mathcal{M}_\mu(\mathbb{F}_q)$ which is
the image of $y$
under the bijection in Proposition \ref{yY}.
We note that the matrix form of $y$ depends on the basis $v_1, v_2, \ldots, v_N$ for $H$.
\end{defi}

Let $\mu \in \lbrace 0,1 \rbrace^N$ with $\mu \neq \mathbf{0}$, $\mathbf{1}$.
For $s \in S_\mu$,
we denote by $s^-$ the one smaller element in $S_\mu$.
If there is no such element, we set $s^- = 0$.
For $t \in T_\mu$,
we denote by $t^+$ the one larger element in $T_\mu$.
If there is no such element, we set $t^+ = N+1$.
Observe that for $(s,t) \in B_\mu$,
we have $(s^-,t) \in B_\mu$ if $s^- \neq 0$
and 
we have $(s,t^+) \in B_\mu$ if $t^+ \neq N+1$.
For $M \in \mathcal{M}_{\mu}(\mathbb{F}_q)$
and for $(s,t) \in B_\mu$,
let $M(s,t)$ denote the submatrix of $M$ indexed by
the rectangle with respect to $(s,t)$ in \eqref{Bmu(s,t)}.
Moreover,
we set
\begin{align}
r^-(M,s,t) &= 
\begin{cases}
\mathrm{rank} \left(M(s^-,t)\right) & \text{if $s^- \neq 0$}, \\
0 & \text{if $s^- = 0$},
\end{cases} \label{r-}\\
r^+(M,s,t) &= 
\begin{cases}
\mathrm{rank} \left(M(s,t^+)\right) & \text{if $t^+ \neq N+1$}, \\
0 & \text{if $t^+ = N+1$},
\end{cases} \label{r+}\\
r^{-+}(M,s,t) &= 
\begin{cases}
\mathrm{rank} \left(M(s^-,t^+)\right) & \text{if $s^- \neq 0$ and $t^+ \neq N+1$}, \\
0 & \text{if $s^- = 0$ or $t^+ = N+1$}.
\end{cases}\label{r-+}
\end{align}

\begin{defi}\label{defi:sigmaM}\normalfont
Let $\mu \in \lbrace 0,1 \rbrace^N$ with $\mu \neq \mathbf{0}$, $\mathbf{1}$.
For $M \in \mathcal{M}_{\mu}(\mathbb{F}_q)$,
we define the set $\sigma(M)$
consisting of all indices $(s,t) \in B_\mu$ 
such that
\[
r^{\epsilon}(M,s,t) = \mathrm{rank} \left(M(s,t)\right) - 1
\]
for all $\epsilon \in \lbrace -, +, -+\rbrace$.
\end{defi}

\begin{lem}\label{lem:RP}
Let $\mu \in \lbrace 0,1 \rbrace^N$ with $\mu \neq \mathbf{0}$, $\mathbf{1}$.
For $M \in \mathcal{M}_\mu(\mathbb{F}_q)$,
the set $\sigma(M)$ in Definition \ref{defi:sigmaM} is a rook placement on $B_\mu$.
\end{lem}

\begin{proof}
Fix $M \in \mathcal{M}_\mu(\mathbb{F}_q)$.
Since $\sigma(M)$ is a subset of $B_\mu$,
it suffices to show that no two elements in $\sigma(M)$ have a common entry.
To do this,
we take $(s_1,t), (s_2, t) \in \sigma(M)$
and assume $s_1 < s_2$.
Observe that $s_2^- \neq 0$.
Since $(s_1,t) \in \sigma(M)$,
we have
\begin{equation}
r^+(M,s_1,t) = \mathrm{rank}\left( M(s_1,t) \right) - 1.
\label{eq:r+(M,s1,t)}
\end{equation}
Since $(s_2,t) \in \sigma(M)$, we have 
$r^{-+}(M,s_2,t) = r^{-}(M,s_2,t)$.
By definition, $r^{-+}(M,s_2,t) = r^+(M,s_2^-,t)$,$r^{-}(M,s_2,t) = \mathrm{rank}\left( M(s_2^-,t) \right)$ and so we obtain
\begin{equation}
r^+(M,s_2^-,t) = \mathrm{rank}\left( M(s_2^-,t) \right).
\label{eq:r+(M,s2-,t)}
\end{equation}
By \eqref{eq:r+(M,s1,t)},
the $t$-th column of $M(s_1,t)$ can't be expressed as a linear combination of other columns of $M(s_1,t)$.
By \eqref{eq:r+(M,s2-,t)},
the $t$-th column of $M(s_2^-,t)$ can be expressed as a linear combination of other columns of $M(s_2^-,t)$.
This implies $s_2^- < s_1$,
which contradicts to $s_1 < s_2$.
Therefore we must have $s_1 = s_2$.
Similarly, 
if we take $(s,t_1), (s, t_2) \in \sigma(M)$, then one can show that $t_1 = t_2$.
So the result follows.
\end{proof}

Recall the local inversion numbers of a rook placement from \eqref{LIN}.

\begin{lem}\label{lem:rankinv}
Let $\mu \in \lbrace 0,1 \rbrace^N$ with $\mu \neq \mathbf{0}$, $\mathbf{1}$.
For $M \in \mathcal{M}_\mu(\mathbb{F}_q)$,
we have
\[
\mathrm{rank} \left( M(s,t) \right) = \mathrm{inv}(\sigma(M),s,t) + 1
\]
for $(s,t) \in \sigma(M)$.
\end{lem}

\begin{proof}
Fix $(s,t) \in \sigma(M)$.
Observe that
$\mathrm{rank} \left( M(s,t) \right)$ can be computed as follows:
\[
\sum_{(s',t') \in B_\mu(s,t)}
\left(
\mathrm{rank} \left( M(s',t') \right)
-
r^-(M,s',t')
-
r^+(M,s',t')
+
r^{-+}(M,s',t')
\right).
\]
Then by the definition of $\sigma(M)$,
each summand is $1$ if $(s',t') \in \sigma(M)$.
We claim that each summand is $0$ if $(s',t') \not\in \sigma(M)$.
Suppose for the moment that the claim is true.
Then
$\mathrm{rank} \left( M(s,t) \right)$
is equal to the cardinality of $\sigma(M) \cap B_\mu(s,t)$.
The result follows from the definition of local inversion numbers.

Therefore, it remains to prove the claim.
If $(s',t') \not\in \sigma(M)$,
then there exists $\epsilon \in \lbrace -, + ,-+\rbrace$
such that $r^\epsilon(M,s',t') \neq \mathrm{rank}\left( M(s',t')\right)-1$.
If $\epsilon = +$, then $r^+(M,s',t') = \mathrm{rank}\left( M(s',t')\right)$.
In this case, 
the $t'$-th column of $M(s',t')$ can be expressed as a linear combination of other columns of $M(s',t')$.
In particular, if $s'^- \neq 0$, 
the $t'$-th column of $M(s'^-,t')$ can be expressed as a linear combination of other columns of $M(s'^-,t')$.
This implies $r^{-+}(M,s',t') = r^{-}(M,s',t')$,
which is also true if $s'^- = 0$.
Therefore, the summand is $0$.
Similarly, if $\epsilon = -$, the summand is $0$.
If $\epsilon = -+$, then we have two possibilities:
$r^{-+}(M,s',t') = \mathrm{rank}\left( M(s',t')\right)$ or
$r^{-+}(M,s',t') = \mathrm{rank}\left( M(s',t')\right)-2$.
For the first case, we have
$\mathrm{rank}\left( M(s',t')\right) = r^-(M,s',t') = r^+(M,s',t') = r^{-+}(M,s',t')$
since 
we have
$\mathrm{rank}\left( M(s',t')\right) \ge r^-(M,s',t') \ge r^{-+}(M,s',t')$
and 
$\mathrm{rank}\left( M(s',t')\right) \ge r^+(M,s',t') \ge r^{-+}(M,s',t')$
by definition.
This also implies the summand is $0$.
For the second case, we have
$\mathrm{rank}\left( M(s',t')\right) = r^-(M,s',t')+1 = r^+(M,s',t')+1 = r^{-+}(M,s',t')+2$
since 
we have
$\mathrm{rank}\left( M(s',t')\right) \le r^-(M,s',t')+1 \le r^{-+}(M,s',t')+2$
and 
$\mathrm{rank}\left( M(s',t')\right) \le r^+(M,s',t')+1 \le r^{-+}(M,s',t')+2$
by definition.
This also implies the summand is $0$.
Hence the claim holds.
\end{proof}

\begin{lem}\label{MandRP}
Let $\mu \in \lbrace 0,1 \rbrace^N$ with $\mu \neq \mathbf{0}$, $\mathbf{1}$.
For a subset $\sigma \subseteq B_\mu$,
the following are equivalent:
\begin{enumerate}
\item there exists $M \in \mathcal{M}_\mu(\mathbb{F}_q)$
such that $\sigma(M) = \sigma$.
\item it is a rook placement on $B_\mu$.
\end{enumerate}
\end{lem}

\begin{proof}
Lemma \ref{lem:RP} shows that (i) implies (ii).

Assume we are given a rook placement $\sigma$ on $B_\mu$.
Consider the matrix $M_\sigma \in \mathcal{M}_\mu(\mathbb{F}_q)$
defined by
\[
(M_\sigma)_{s,t} =
\begin{cases}
1 & \text{if $(s,t) \in \sigma$}, \\
0 & \text{otherwise}
\end{cases}
\]
for $s \in S_\mu$, $t \in T_\mu$.
Then it is easy to check that $\sigma(M) = \sigma$.
So (ii) implies (i).
\end{proof}

\section{The number of matrices with given parameter}\label{sec:number}

Let $\mu \in \lbrace 0,1 \rbrace^N$ with $\mu \neq \mathbf{0}, \mathbf{1}$.
Recall from Lemma \ref{MandRP}
that
each matrix $\mathcal{M}_\mu(\mathbb{F}_q)$
corresponds to a rook placement on the Ferrers board $B_\mu$ of shape $\mu$.
Recall the sets from \eqref{STmu} and \eqref{STmu(s,t)}.
To simplify the notation, we set
\begin{equation}
n(\pi_1) = \sum_{s \in \pi_1}|S_\mu(s)|
\label{npi1}
\end{equation}
for a subset $\pi_1 \subseteq S_\mu$.

\begin{defi}\label{columnfull}\normalfont
Let $\mu \in \lbrace 0,1 \rbrace^N$.
A subset $\lambda \subseteq \lbrace 1, 2, \ldots, N \rbrace$ is said to be 
\emph{column-full} with respect to $\mu$
whenever $T_\mu \subseteq \lambda$.
Moreover, 
a rook placement $\sigma$ on $B_\mu$
is said to be \emph{column-full}
whenever 
the type of $\sigma$ is column-full.
\end{defi}

Let $\mu \in \lbrace 0,1 \rbrace^N$.
We remark that
a rook placement $\sigma$ on $B_\mu$
is column-full if and only if
the column index set $\pi_2(\sigma)$, defined in \eqref{pisigma}, is maximal.

\begin{prop}\label{numbersigma}
Let $\mu \in \lbrace 0,1 \rbrace^N$ with $\mu \neq \mathbf{0}, \mathbf{1}$
and let $\sigma$ denote a rook placement on $B_\mu$.
Assume $\sigma$ is column-full in Definition \ref{columnfull}.
Then the number of matrices 
$M \in \mathcal{M}_\mu(\mathbb{F}_q)$
such that $\sigma = \sigma(M)$ in Definition \ref{defi:sigmaM} is given by
\[
(q-1)^{|\mu|} q^{\mathrm{inv}(\sigma) + |B_\mu| - n(\pi_1(\sigma))}.
\]
\end{prop}

\begin{proof}
Let $t \in T_\mu$.
We count the number of possibilities for the $t$-th column of $M$ with $\sigma = \sigma(M)$.
Since $\sigma$ is a column-full rook placement, there uniquely exists $s \in S_\mu$
such that $(s,t) \in \sigma$.
Since $(s,t) \in \sigma$, we have
\begin{equation}\label{lem5.2_1}
r^-(M,s,t) = r^{-+}(M,s,t) = \mathrm{rank}\left( M(s,t)\right)-1.
\end{equation}
This means that the $t$-th column of the submatrix $M(s^-,t)$ is a linear combination of other columns.
Therefore, the number of possibilities for the $t$-th column of $M(s^-,t)$ is $q^{r(M,s,t)-1}$.
For a given such column of $M(s^-,t)$, the number of possibilities for the $t$-th column of $M(s,t)$ is at most $q$ since $M(s,t)$ has one more row than $M(s^-,t)$.
In other words, the number of possibilities for the $t$-th column of $M(s,t)$ is at most $q^{r(M,s,t)-1} \times q = q^{r(M,s,t)}$.
Similarly, since $(s,t) \in \sigma$, we have 
\begin{equation}\label{lem5.2_2}
r(M,s,t)-1 = r^+(M,s,t).
\end{equation}
This means that the $t$-th column of the submatrix $M(s,t)$ is not a linear combination of other columns.
Since there are $q^{r(M,s,t)-1}$ columns which are linear combinations of columns of $M(s,t^+)$,
the number of possibilities for the $t$-th column of $M(s,t)$ is
\[
q^{r(M,s,t)}-q^{r(M,s,t)-1} = (q-1)q^{r(M,s,t)-1} = (q-1)q^{\mathrm{inv}(\sigma,s,t)}.
\]
The second equality follows from Lemma \ref{lem:rankinv}.
Since $M \in \mathcal{M}_\mu(\mathbb{F}_q)$, or equivalently $\mathrm{Supp}(M) \subseteq B_\mu$,
the $(s',t)$-entries are $0$ if \textcolor{red}{$s'>t$}.
Therefore, for a given $t$-th column of $M(s,t)$,
the number of possibilities for the $t$-th column of $M$ is at most $q^l$, where
\[
l = |\lbrace s' \in S_\mu \mid s < s' \le t \rbrace|
=  |S_\mu(t)| - |S_\mu(s)|.
\]
Observe that any choices of the $t$-th column among the $q^l$ possibilities satisfy both \eqref{lem5.2_1} and \eqref{lem5.2_2} by construction.
Since the conditions \eqref{lem5.2_1} and \eqref{lem5.2_2} are equivalent to $(s,t) \in \sigma$,
the number is exactly $q^l$.
We have shown that the number of possibilities for the $t$-th column of $M$ is
\[
(q-1)q^{\mathrm{inv}(\sigma,s,t)} \times q^{|S_\mu(t)| - |S_\mu(s)|},
\]
which is independent of the choice of other columns of $M$.
Therefore the number of $M$ is obtained by taking the product of the values for all $t \in T_\mu$ since $\sigma$ is column-full. 
The result follows from
the definition of $\mathrm{inv}(\sigma)$
and
the column-full property
and
\[
\sum_{t \in T_\mu}|S_\mu(t)| = |\lbrace (s,t) \in S_\mu \times T_\mu \mid s < t \rbrace| = |B_\mu|.
\]
\end{proof}

\begin{cor}\label{cor:numberM}
Let $\mu \in \lbrace 0,1 \rbrace^N$ with $\mu \neq \mathbf{0}, \mathbf{1}$
and
let
$\lambda \subseteq \lbrace 1, 2, \ldots, N\rbrace$ 
satisfy (ii) in Lemma \ref{LSmuLTmu}.
Assume $\lambda$ is column-full with respect to $\mu$ in Definition \ref{columnfull}.
Then the number of matrices 
$M \in \mathcal{M}_\mu(\mathbb{F}_q)$
such that $\sigma(M)$ is of type $\lambda$
in Definitions \ref{sigmatype} and \ref{defi:sigmaM} is given by
\[
q^{|B_\mu| - n(\lambda \cap S_\mu)}
\prod_{s \in \lambda \cap S_\mu} \left(q^{\rho(s,\mu,\lambda)} - 1\right),
\]
where $\rho(s,\mu,\lambda)$ is defined in Lemma \ref{qinvsigma}.
\end{cor}

\begin{proof}
Use Lemma \ref{qinvsigma} and Proposition \ref{numbersigma}.
\end{proof}

\section{The algebra $\mathcal{H}$}\label{sec:H}

Recall $\mathrm{Mat}_P(\mathbb{C})$,
the set of all matrices
whose rows and columns are indexed by $P$ and whose entries are in $\mathbb{C}$.
We see it as a $\mathbb{C}$-algebra.
We write $I \in \mathrm{Mat}_P(\mathbb{C})$ for the identity matrix
and $O \in \mathrm{Mat}_P(\mathbb{C})$ for the zero matrix.
In this section, 
we introduce a subalgebra
$\mathcal{H}$ of $\mathrm{Mat}_P(\mathbb{C})$
which represents the $N$-cube structure in $P$.

Let $V = \mathbb{C}P$ denote the vector space over $\mathbb{C}$
consisting of the column vectors whose coordinates are indexed by $P$
and whose entries are in $\mathbb{C}$.
Observe that $\mathrm{Mat}_P(\mathbb{C})$
acts on $V$ by left multiplication.
We call $V$ the \emph{standard module}
for $\mathrm{Mat}_P(\mathbb{C})$.
We equip $V$ with the standard Hermitian inner product defined by 
$\langle u, v \rangle = u^T\bar{v}$
for $u, v \in V$,
where $^T$ denotes transpose and $\bar{}$ denotes complex conjugate.

Recall from Definition \ref{defi:location}
that we have partitioned $P$ into the sets $P_\mu$ of all subspaces at location $\mu$ for $\mu \in \lbrace 0,1 \rbrace^N$.
For $\mu \in \mathbb{Z}^N$, define a diagonal matrix $E_\mu^* \in \mathrm{Mat}_P(\mathbb{C})$ by
\begin{align*}
(E_\mu^*)_{y,y} =
\begin{cases}
1 & \text{if $y \in P_\mu$}, \\
0 & \text{if $y \not\in P_\mu$},
\end{cases}
&&
y \in P.
\end{align*}
Observe that
$E_\mu^* = O$
unless $\mu \in \lbrace 0,1 \rbrace^N$.
By construction, we have 
\begin{align*}
E_\mu^* E_\nu^* = \delta_{\mu,\nu} E_\mu^*,
&&
\mu, \nu \in \lbrace 0,1 \rbrace^N,
\end{align*}
\[
I = \sum_{\mu \in \lbrace 0,1 \rbrace^N} E_\mu^*.
\]
Moreover, we have a decomposition of $V$:
\begin{align*}
V = \sum_{\mu \in \lbrace 0,1 \rbrace^N} E_\mu^* V,
&&
\text{(direct sum)},
\end{align*}
where $E_\mu^* V$ is the subspace of $V$
consisting of the vectors whose nonzero entries are indexed by elements in $P_\mu$.
Thus, the matrix $E_\mu^*$ is the projection from $V$ onto $E_\mu^* V$
and
we call it the \emph{projection matrix}.

\begin{defi}\normalfont\label{K}
By the above comments, the matrices
$E_\mu^*$, where $\mu \in \lbrace 0,1 \rbrace^N$ form a basis
for a commutative subalgebra of $\mathrm{Mat}_P(\mathbb{C})$.
We denote this subalgebra by $\mathcal{K}$.
\end{defi}

We now introduce matrices that generate $\mathcal{K}$.
For $1 \le m \le N$, we define diagonal matrices $K_m \in \mathrm{Mat}_P(\mathbb{C})$ by
\begin{align*}
(K_m)_{y,y} = q^{1/2 - \mu_m},
&&
y \in P_{\mu},
\end{align*}
where $\mu = (\mu_1, \mu_2, \ldots, \mu_N)$.

\begin{lem}\label{lem:K}
For $1 \le m \le N$,
we have
\[
K_m = \sum_{\mu \in \lbrace 0,1 \rbrace^N}q^{1/2 - \mu_m} E_\mu^*,
\]
where $\mu = (\mu_1, \mu_2, \ldots, \mu_N)$.
\end{lem}

\begin{proof}
Immediate from the construction.
\end{proof}

\begin{prop}
The algebra $\mathcal{K}$ in Definition \ref{K} is generated by $K_m$ for $1 \le m \le N$.
\end{prop}

\begin{proof}
By Lemma \ref{lem:K}, the matrices $K_m$ $(1 \le m \le N)$ generate a subalgebra
$\mathcal{K}'$ of $\mathcal{K}$.
By Lemma \ref{lem:K} and since $E_\mu^*$ are idempotent, for $\nu = (\nu_1, \nu_2, \ldots, \nu_N) \in \lbrace 0,1 \rbrace^N$, we have
\[
K_1^{\nu_1}K_2^{\nu_2} \cdots K_N^{\nu_N}
=
\sum_{\mu \in \lbrace 0,1 \rbrace^N}
q^{\sum_{m=1}^N(\nu_m/2 - \mu_m\nu_m)}E_\mu^*.
\]
By linear algebra, if the coefficient matrix $Q_N$ indexed by $\lbrace 0,1 \rbrace^N$,
whose $(\nu, \mu)$-entry is $q^{\sum_{m=1}^N(\nu_m/2 - \mu_m\nu_m)}$, is invertible,
then each $E_\mu^*$ is a linear combination of
$K_1^{\nu_1}K_2^{\nu_2} \cdots K_N^{\nu_N}$ $(\nu = (\nu_1, \nu_2, \ldots, \nu_N) \in \lbrace 0,1 \rbrace^N)$.
In particular, in this case, $E_\mu^*$ is a polynomial in $K_m$ $(1 \le m \le N)$
for every $\mu \in \lbrace 0,1 \rbrace^N$ and
consequently, $\mathcal{K}' = \mathcal{K}$.
So, it remains to show that the determinant of $Q_N$ is nonzero.
First, 
observe that if $N=1$, 
\[
\det Q_1 = \det \begin{pmatrix}1&1\\q^{1/2}&q^{-1/2}\end{pmatrix} = q^{-1/2} - q^{1/2} \neq 0
\]
since $q \neq 1$.
We next consider the matrix $Q_1^{\otimes N}$ indexed by $\lbrace 0,1 \rbrace^N$.
The $(\nu, \mu)$-entry of $Q_1^{\otimes N}$ is given by
\[
q^{\sum_{m=1}^N(\nu_m/2 - \mu_m\nu_m)},
\]
which is same as that of $Q_N$.
This means $Q_N = Q_1^{\otimes N}$.
By $Q_N = Q_1^{\otimes N}$ and $\det Q_1 \neq 0$, we conclude that $\det Q_N \neq 0$ as desired.
\end{proof}

Next we introduce two kinds of matrices from covering relations in Definition \ref{mcovers}.
For $1 \le m \le N$,
the matrices $L_m, R_m \in \mathrm{Mat}_P(\mathbb{C})$ are defined by
\begin{align*}
(L_m)_{y,z} = \begin{cases}
1 & \text{if $z$ $m$-covers $y$}, \\
0 & \text{otherwise},
\end{cases}
&&
(R_m)_{y,z} = \begin{cases}
1 & \text{if $y$ $m$-covers $z$}, \\
0 & \text{otherwise}
\end{cases}
\end{align*}
for $y,z \in P$.
We remark that for each $1 \le m \le N$,
the matrices $L_m$ and $R_m$ are transposes of one another.
Recall the comment in the above of Lemma \ref{lem:adj}.

\begin{lem}\label{lem:actLmRm}
For $1 \le m \le N$ and
$\mu \in \lbrace 0,1 \rbrace^N$,
we have the following.
\begin{enumerate}
\item $L_mE_\mu^* = E_{\mu - \widehat{m}}^*L_m$ and 
$R_mE_\mu^* = E_{\mu + \widehat{m}}^*R_m$.
\item $L_mE_\mu^* V \subseteq E_{\mu - \widehat{m}}^* V$ and 
$R_mE_\mu^* V \subseteq E_{\mu + \widehat{m}}^* V$.
\end{enumerate}
\end{lem}

\begin{proof}
Immediate from the construction.
\end{proof}

Because of Lemma \ref{lem:actLmRm} (ii),
we call $L_m$ the \emph{lowering matrices}
and $R_m$ the \emph{raising matrices}.

\begin{defi}\normalfont\label{H}
Let $\mathcal{H}$ denote the subalgebra of $\mathrm{Mat}_P(\mathbb{C})$
generated by $L_m$, $R_m$ $(1 \le m \le N)$ and the algebra $\mathcal{K}$ in Definition \ref{K}.
\end{defi}

\begin{prop}\label{semisimple}
The algebra $\mathcal{H}$ in Definition \ref{H} is semisimple.
\end{prop}

\begin{proof}
This follows since $\mathcal{H}$ is closed under the conjugate-transpose map.
\end{proof}

We recall 
the incidence algebra, which is generated by $L$, $R$ and $E_i^\star$ $(0 \le i \le N)$
from the second paragraph in Section \ref{introduction}.
We remark that $\mathcal{H}$ contains the incidence algebra  
as its subalgebra
because
$L = \sum_{m = 1}^N L_m$,
$R = \sum_{m = 1}^N R_m$
and
$E_i^\star = \sum_{\mu \in \lbrace 0,1\rbrace^N, |\mu| = i} E_\mu^*$. 
Moreover,
if $N \ge 2$, 
the incidence algebra is a proper subalgebra of $\mathcal{H}$.

\section{The structure of the algebra $\mathcal{H}$}\label{sec:strH}

In this section, we discuss the relations among the generators
$L_m$, $R_m$, $K_m$ of the algebra $\mathcal{H}$.

\begin{prop}\label{prop:relations2}
For $1 \le m,n \le N$ with $m \neq n$, the following hold.
\begin{enumerate}
\item $L_mK_n = K_nL_m$.
\item $R_mK_n = K_nR_m$.
\item $qL_mK_m = K_mL_m$.
\item $R_mK_m = qK_mR_m$.
\end{enumerate}
\end{prop}

\begin{proof}
This lemma follows by combining Lemmas \ref{lem:K} and \ref{lem:actLmRm} (i).
\end{proof}

\begin{prop}\label{prop:relations1}
For $1 \le m, n \le N$, we have the following.
\begin{enumerate}
\item $L_m^2 = R_m^2 = 0$.
\item $qL_mL_n = L_nL_m$ if $m < n$.
\item $R_mR_n = qR_nR_m$ if $m < n$.
\item $L_mR_n = R_nL_m$ if $m \neq n$.
\end{enumerate}
\end{prop}

\begin{proof}
(i) It follows from the definition of $L_m$ and $R_m$.
(ii), (iii) These are matrix reformulations of Lemma \ref{lem:count} (i), (ii).
(iv) This is a matrix reformulation of Lemma \ref{lem:count} (iii), (iv).
\end{proof}

\section{The $L_m$- and $R_m$-actions on $V$}\label{LRactiononV}

We now describe a basis for $V$,
which is the key in this paper.
In this section, we fix a basis $v_1, v_2, \ldots, v_N$ for $H$
adapted to the flag $\lbrace x_i \rbrace_{i=0}^N$
and assume that the matrix forms in Definition \ref{matrixform}
are always taken with respect to this basis $v_1, v_2, \ldots, v_N$.

\begin{defi}\label{chiy}\normalfont
Let $\chi$ denote a nontrivial character of the additive group $\mathbb{F}_q$
and let $\mu \in \lbrace 0,1 \rbrace^N$.
For $y \in P_\mu$,
define a vector $\chi_y \in V$ as follows.
\begin{enumerate}
\item If $\mu = \mathbf{0}$ or $\mathbf{1}$,
then for $z \in P$,
the $z$-th entry of $\chi_y$ is $1$ if $y = z$ and $0$ otherwise.
\item If $\mu \neq \mathbf{0}, \mathbf{1}$,
then for $z \in P$,
the $z$-th entry of $\chi_y$ is defined by
\[
\begin{cases}
\chi\left( \mathrm{tr}(YZ^T) \right) & \text{if $z \in P_\mu$}, \\
0 & \text{if $z \not\in P_\mu$},
\end{cases}
\]
where $Y, Z \in \mathcal{M}_\mu(\mathbb{F}_q)$
are the matrix forms of $y, z$, respectively in Definition \ref{matrixform}.
Here $^T$ denotes transpose and $\mathrm{tr}$ denotes the trace map of matrices.
\end{enumerate}
\end{defi}

For the rest of this section,
we fix a nontrivial character $\chi$ of the additive group $\mathbb{F}_q$.

\begin{lem}\label{lem:basis}
For $\mu \in \lbrace 0,1 \rbrace^N$,
the set of vectors $\chi_y \in V$ for $y \in P_\mu$ in Definition \ref{chiy}
forms an orthogonal basis for the vector space $E_\mu^* V$.
\end{lem}

\begin{proof}
Let $\mu \in \lbrace 0,1 \rbrace^N$.
For $y \in P_\mu$,
observe that $\chi_y \in E_\mu^* V$ from the construction.
If $\mu = \mathbf{0}$ or $\mathbf{1}$,
then the assertion is trivial, since $\dim E_\mu^* V = 1$.
Assume $\mu \neq \mathbf{0}, \mathbf{1}$ and take $y, y' \in P_\mu$.
Consider the Hermitian inner product
\[
\langle \chi_y, \chi_{y'} \rangle = \sum_{z \in P} \chi_y(z)\overline{\chi_{y'}(z)},
\]
where $\chi_y(z)$, $\chi_{y'}(z)$ denote the $z$-th entries of $\chi_y$, $\chi_{y'}$, respectively.
By the definitions of $\chi_y(z)$, $\chi_{y'}(z)$, we have
\[
\langle \chi_y, \chi_{y'} \rangle = \sum_{Z \in \mathcal{M}_\mu(\mathbb{F}_q)}\chi\left(\mathrm{tr} (Y-Y')Z^T \right),
\]
where $Y, Y'$ are the matrix forms of $y,y'$, respectively.
Assume $y \neq y'$ and equivalently $Y \neq Y'$.
Observe that for $g \in \mathbb{F}_q$,
the number of $Z \in \mathcal{M}_\mu(\mathbb{F}_q)$ such that $\mathrm{tr} (Y-Y')Z^T = g$
does not depend on $g$,
and so the number is $|\mathcal{M}_\mu(\mathbb{F}_q)|/|\mathbb{F}_q| = q^{|B_\mu|-1}$.
Therefore, we have
\[
\langle \chi_y, \chi_{y'} \rangle = q^{|B_\mu|-1} \sum_{g \in \mathbb{F}_q}\chi(g) = 0.
\]
The last equality follows from the orthogonality of the character $\chi$ and the trivial character.
Therefore the set of vectors $\chi_y$ for $y \in P_\mu$
becomes an orthogonal basis for a subspace $V_\mu$ of $E_\mu^* V$.
By comparing their dimensions,
we have $V_\mu = E_\mu^* V$ and the result follows.
\end{proof}

Recall the $m$-covering relation from Definition \ref{mcovers}.

\begin{lem}\label{ymcoversz}
Let $1 \le m \le N$
and
let $\mu, \nu \in \lbrace 0,1 \rbrace^N$ with $\mu, \nu \neq \mathbf{0}, \mathbf{1}$ such that $\mu$ $m$-covers $\nu$.
Take $y \in P_\mu$, $z \in P_\nu$ and let
$Y \in \mathcal{M}_\mu(\mathbb{F}_q)$ and $Z \in \mathcal{M}_{\nu}(\mathbb{F}_q)$
denote the matrix forms of $y, z$, respectively in Definition \ref{matrixform}.
Then $y$ $m$-covers $z$ if and only if
\[
Z_{s,t} = Y_{s,t} + Y_{s,m}Z_{m,t}
\]
for $s \in S_\mu$ and for $t \in T_{\nu}$.
\end{lem}

\begin{proof}
Recalling the bijection of Proposition \ref{yY},
for $t \in T_\mu$ and $t' \in T_{\nu}$,
we write
\begin{align*}
w_t(Y) = \sum_{s \in S_\mu} Y_{s,t}v_s + v_t,
&&
w_{t'}(Z) = \sum_{s' \in S_{\nu}} Z_{s',t'}v_{s'} + v_{t'}.
\end{align*}
Assume $y$ covers $z$.
For each $t' \in T_\nu$,
since $z \subseteq y$, 
the vector $w_{t'}(Z)$ is a linear combination of $w_t(Y)$, where $t \in T_\mu$.
Comparing the coefficients of $v_t$ for $t \in T_\mu$,
we have $w_{t'}(Z) = Z_{m,t'}w_m(Y) + w_{t'}(Y)$.
Then comparing the coefficients of $v_s$ for $s \in S_\mu$, we obtain the desired equality.
On the other hand, assume the equality $Z_{s,t'} = Y_{s,t'} + Z_{m,t'}Y_{s,m}$
for $s \in S_\mu$ and $t' \in T_\nu$.
By the same argument above, we have $w_{t'}(Z) \in y$ for all $t' \in T_\nu$.
This implies $y$ covers $z$, as desired.
\end{proof}

\begin{lem}\label{lem:Lmchiy}
Let $1 \le m \le N$
and
let $\mu, \nu \in \lbrace 0,1 \rbrace^N$ with $\mu, \nu \neq \mathbf{0}, \mathbf{1}$ such that $\mu$ $m$-covers $\nu$.
Take $y \in P_\mu$, $z \in P_\nu$ and
let $Y \in \mathcal{M}_\mu(\mathbb{F}_q)$, $Z \in \mathcal{M}_{\nu}(\mathbb{F}_q)$ denote the matrix forms of $y, z$, respectively in Definition \ref{matrixform}.
Then the $z$-th entry of $L_m\chi_y$
is given by
\[
L_m\chi_y(z) = q^{|S_\mu(m-1)|}\chi\left(\sum_{s \in S_\mu}\sum_{t \in T_\nu} Y_{s,t}Z_{s,t}\right)
\]
if $Y_{s,m} = \sum_{t \in T_\nu} Y_{s,t}Z_{m,t}$
for all $s \in S_\mu$ with $s < m$
and $0$ otherwise.
\end{lem}

\begin{proof}
By the definition of $L_m$,
the $z$-th entry of $L_m\chi_y$ is defined by
\[
L_m\chi_y(z) = \sum_{y'} \chi_y(y'),
\]
where 
the sum is taken over all $y' \in P_\mu$ such that $y'$ $m$-covers $z$.
Then by Definition \ref{chiy} and Lemma \ref{ymcoversz},
we have
\[
L_m\chi_y(z) 
=
\sum_{Y'}\chi\left( \sum_{s \in S_\mu}\sum_{t \in T_\mu} Y_{s,t}Y'_{s,t}\right),
\]
where the sum is taken over all $Y' \in \mathcal{M}_\mu(\mathbb{F}_q)$ such that $Z_{s,t}=Y'_{s,t}+Y'_{s,m}Z_{m,t}$ for $s \in S_\mu$ and for $t \in T_\nu$.
Observe that $T_\mu \setminus T_{\nu} = \{m\}$ and so we have
\begin{align*}
L_m\chi_y(z) &= 
\sum\chi\left( \sum_{s \in S_\mu}\sum_{t \in T_\nu} Y_{s,t}(Z_{s,t} - Y'_{s,m}Z_{m,t}) + \sum_{s \in S_\mu} Y_{s,m}Y'_{s,m}\right) \\
&= 
\sum\chi\left( \sum_{s \in S_\mu}\sum_{t \in T_\nu} Y_{s,t}Z_{s,t}\right)\chi\left( \sum_{s \in S_\mu}\left( Y_{s,m} - \sum_{t \in T_\nu} Y_{s,t}Z_{m,t}\right)Y'_{s,m}\right),
\end{align*}
where the first sum in each line is taken over all $Y'_{s,m} \in \mathbb{F}_q$ such that $Y'_{s,m} = 0$ if $s > m$.
If $Y_{s,m} \neq \sum_{t \in T_\nu} Y_{s,t}Z_{m,t}$ for some $s \in S_\mu$ with $s < m$, then by the same argument as in the proof of Lemma \ref{lem:basis}, the sum is $0$.
If $Y_{s,m} = \sum_{t \in T_\nu} Y_{s,t}Z_{m,t}$ for all $s \in S_\mu$ with $s < m$, then
\[
L_m\chi_y(z) =
q^{|S_\mu(m-1)|}
\chi\left( \sum_{s \in S_\mu}\sum_{t \in T_\nu} Y_{s,t}Z_{s,t}\right).
\]
Here the coefficient $q^{|S_\mu(m-1)|}$ is the number of choices for $Y'_{s,m} \in \mathbb{F}_q$ for $s \in S_\mu$ with $s < m$.
The result follows.
\end{proof}

\begin{lem}\label{lem:Rmchiz}
Let $1 \le m \le N$
and
let $\mu, \nu \in \lbrace 0,1 \rbrace^N$ with $\mu, \nu \neq \mathbf{0}, \mathbf{1}$ such that $\mu$ $m$-covers $\nu$.
Take $y \in P_\mu$, $z \in P_\nu$ and
let $Y \in \mathcal{M}_\mu(\mathbb{F}_q)$, $Z \in \mathcal{M}_{\nu}(\mathbb{F}_q)$ denote the matrix forms of $y, z$, respectively in Definition \ref{matrixform}.
Then the $y$-th entry of $R_m\chi_z$
is given by
\[
R_m\chi_z(y) = q^{|T_\nu(m+1)|}\chi\left(\sum_{s \in S_\mu}\sum_{t \in T_\nu} Y_{s,t}Z_{s,t}\right)
\]
if $Z_{m,t} = - \sum_{s \in S_\mu} Z_{s,t}Y_{s,m}$
for all $t \in T_\nu$ with $t > m$
and $0$ otherwise.
\end{lem}

\begin{proof}
Similar to the proof of Lemma \ref{lem:Lmchiy}.
\end{proof}

\begin{lem}\label{numberofZ}
Referring to Lemma \ref{lem:Lmchiy},
let $\lambda$ denote the type of $\sigma(Y)$
in Definitions \ref{sigmatype} and \ref{defi:sigmaM}.
Then
the number of $Z \in \mathcal{M}_\nu(\mathbb{F}_q)$ such that 
$Y_{s,m} = \sum_{t \in T_\nu} Y_{s,t}Z_{m,t}$
for all $s \in S_\mu$ with $s < m$
is given by $q^l$ where
\[
l = |B_\nu| - |\lambda \cap S_\mu(m-1)| - |\lambda \cap T_\mu(m+1)| + |\lambda|/2
\]
if $m \not\in \lambda$,
and $0$ otherwise.
\end{lem}

\begin{proof}
We count the number of possibilities for $Z_{s,t} \in \mathbb{F}_q$
for $s \in S_\nu$ and $t \in T_\nu$.
If $s > t$, then $Z_{s,t} = 0$
since $\mathrm{Supp}(Z) \subseteq B_\nu$.
If $s \neq m$ and $s < t$, then $Z_{s,t}$ is arbitrary
and therefore the number of possibilities is $q$.
The number of such pairs $(s,t)$ is given by
\[
|\lbrace (s,t) \in B_\nu \mid s \neq m \rbrace|
=
|B_\nu| - |T_\nu(m+1)|.
\]
For the case $s = m$ and $m < t$,
by the constraint,
the sequence $(Z_{m,t})_{t \in T_\nu, t>m}$
must be a solution of the system of linear equations over $\mathbb{F}_q$:
\[
C\mathbf{u} = \mathbf{c},
\]
where $C = (Y_{s,t})_{s \in S_\mu, s<m, t \in T_\nu, t > m}$
is the coefficient matrix,
$\mathbf{u} = (u_t)_{t \in T_\nu, t > m}$
is the unknown vector
and
$\mathbf{c} = (Y_{s,m})_{s \in S_\mu, s < m}$
is the constant vector.
By linear algebra,
the system $C\mathbf{u} = \mathbf{c}$ has a solution
if and only if
the rank of the augmented matrix $[C, \mathbf{c}]$
is equal to the rank of the coefficient matrix $C$. 
By Definition \ref{defi:sigmaM}, 
it is also equivalent to $(s,m) \not\in \sigma(Y)$ for all $s \in S_\mu$ with
$s < m$,
which means $m \not\in \lambda$.
Moreover,
suppose there is a solution of the system $C\mathbf{u} = \mathbf{c}$.
Since there are $|T_\nu(m+1)|$ columns in $C$,
the number of solutions
is given by
\[
q^{|T_\nu(m+1)| - \rank C}.
\]
By the proof of Lemma \ref{lem:rankinv}, the rank of $C$ is computed as follows:
\begin{align*}
\rank C
&= |\lbrace (s,t) \in \sigma(Y) \mid s \le m-1, t \ge m+1 \rbrace| \\
&= |\lbrace (s,t) \in \sigma(Y) \mid s \le m-1 \rbrace| + |\lbrace (s,t) \in \sigma(Y) \mid t \ge m +1 \rbrace| - |\sigma(Y)|\\
&= |\lambda \cap S_\mu(m-1)| + |\lambda \cap T_\mu(m+1)| - |\lambda|/2.
\end{align*}
Therefore the result follows.
\end{proof}

\begin{lem}\label{numberofY}
Referring to Lemma \ref{lem:Rmchiz},
let $\lambda$ denote the type of $\sigma(Z)$
in Definitions \ref{sigmatype} and \ref{defi:sigmaM}.
Then
the number of $Y \in \mathcal{M}_\mu(\mathbb{F}_q)$ such that 
$Z_{m,t} = - \sum_{s \in S_\mu} Z_{s,t}Y_{s,m}$
for all $t \in T_\nu$ with $t > m$
is given by $q^l$ where
\[
l = |B_\mu| - |\lambda \cap S_\mu(m-1)| - |\lambda \cap T_\mu(m+1)| + |\lambda|/2
\]
if $m \not\in \lambda$,
and $0$ otherwise.
\end{lem}

\begin{proof}
Similar to the proof of Lemma \ref{numberofZ}.
\end{proof}

\begin{defi}\label{type}\normalfont
Let $\mu \in \lbrace 0,1 \rbrace^N$ and take $y \in P_\mu$.
If $\mu \neq \mathbf{0}, \mathbf{1}$,
then
let $Y \in \mathcal{M}_\mu(\mathbb{F}_q)$ denote the matrix form of $y$
in Definition \ref{matrixform}.
Then the \emph{type} of $y$
is defined to be the type of $\sigma(Y)$
in Definitions \ref{sigmatype} and \ref{defi:sigmaM}.
If $\mu = \mathbf{0}$ or $\mathbf{1}$,
then the \emph{type} of $y$
is defined to be the empty set.
We note that the type of $y$ depends on the basis $v_1, v_2, \ldots, v_N$ for $H$
since the matrix form does.
\end{defi}

\begin{lem}\label{Lmchiy0}
Let $\mu \in \lbrace 0,1 \rbrace^N$
and let $\lambda \subseteq \lbrace 1, 2, \ldots, N\rbrace$ satisfy (ii) in Lemma \ref{LSmuLTmu}.
For $1 \le m \le N$,
the following are equivalent:
\begin{enumerate}
\item for any $y \in P_\mu$ of type $\lambda$, we have $L_m\chi_y = 0$;
\item $m \in S_\mu$ or $m \in \lambda$.
\end{enumerate}
\end{lem}

\begin{proof}
Set $\nu = \mu - \widehat{m}$ so that $\mu$ $m$-covers $\nu$.
Then $\nu \neq \mathbf{1}$.
For $y \in P_\mu$, 
observe that $L_m\chi_y \in E_\nu^*V$ by Lemma \ref{lem:actLmRm} (ii).

(i) $\Rightarrow$ (ii)
Suppose $L_m\chi_y = 0$ for any $y \in P_\mu$ of type $\lambda$.
If $\mu = \mathbf{0}$, then $m \in S_\mu$ and so (ii) holds.
If $\mu = \mathbf{1}$, then $P_\mu = \lbrace y=H \rbrace$ and any subspaces $z \in P_\nu$ are $m$-covered by $y$,
and so the $z$-th entry of $L_m\chi_y$ is 
\[
L_m\chi_y(z) = \chi_y(y) = 1
\]
by Definition \ref{chiy}.
This is a contradiction to $L_m\chi_y = 0$.
If $\nu = \mathbf{0}$, then $P_\nu = \lbrace 0 \rbrace$ and any subspaces $y' \in P_\mu$ $m$-cover $0$,
and so the $0$-th entry of $L_m\chi_y$ is 
\[
L_m\chi_y(0) = \sum_{Y' \in \mathcal{M}_\mu(\mathbb{F}_q)}\chi\left( \mathrm{tr}(YY'^T)\right),
\]
where $Y$ is the matrix form of $y$.
By the same argument as in the proof of Lemmas \ref{lem:basis} and \ref{lem:Lmchiy},
the sum vanishes (if and) only if $Y$ is not the zero matrix from the orthogonality of the characters $\chi$ and the trivial character.
Since $y \in P_{\widehat{m}}$, we must have $m \in \lambda$.
If $\mu, \nu \neq \mathbf{0}, \mathbf{1}$,
then by Lemma \ref{lem:Lmchiy},
$L_m\chi_y=0$ implies that there is no $Z \in \mathcal{M}_{\nu}(\mathbb{F}_q)$
such that $Y_{s,m} = \sum_{t \in T_\nu}Y_{s,t}Z_{m,t}$ for all $s \in S_\mu$ with $s < m$,
where $Y \in \mathcal{M}_{\mu}(\mathbb{F}_q)$ denote the matrix form of $y$ in Definition \ref{matrixform}.
In this case, by Lemma \ref{numberofZ}, we have $m \in \lambda$, where $\lambda$ is the type of $y$.

(ii) $\Rightarrow$ (i)
Suppose $m \in S_\mu$ or $m \in \lambda$.
If $m \in S_\mu$, then $\nu \not\in \lbrace 0,1 \rbrace^N$ and so $E_\mu^*V = 0$.
This implies $L_m\chi_y = 0$ since $L_m\chi_y \in E_\nu^*V$.
We now assume $m \in T_\mu$ and $m \in \lambda$.
Observe that $\mu \neq \mathbf{0}$.
If $\mu = \mathbf{1}$, then $\lambda = \emptyset$ by Definition \ref{type}.
This contradicts to $m \in \lambda$.
If $\nu = \mathbf{0}$, then by the similar argument above, $m \in \lambda$ implies the matrix form of $y$ is not the zero matrix.
Then this implies the $0$-th entry of $L_m\chi_y$ is $0$, which means $L_m\chi_y = 0$.
If $\mu, \nu \neq \mathbf{0}, \mathbf{1}$,
then the result follows from Lemmas \ref{lem:Lmchiy} and \ref{numberofZ}.
\end{proof}

\begin{lem}\label{Rmchiy0}
Let $\nu \in \lbrace 0,1 \rbrace^N$
and let $\lambda \subseteq \lbrace 1, 2, \ldots, N\rbrace$ satisfy (ii) in Lemma \ref{LSmuLTmu} with $\mu$ replaced by $\nu$.
For $1 \le m \le N$,
the following are equivalent:
\begin{enumerate}
\item for any $z \in P_\nu$ of type $\lambda$, we have $R_m\chi_z = 0$;
\item $m \in T_\nu$ or $m \in \lambda$.
\end{enumerate}
\end{lem}

\begin{proof}
Similar to the proof of Lemma \ref{Lmchiy0}.
\end{proof}

Recall from Lemma \ref{cardeven},
a subset $\lambda \subseteq \lbrace 1, 2, \ldots, N \rbrace$ becomes a type
if and only if it has even cardinality.
For $\lambda \subseteq \lbrace 1, 2, \ldots, N \rbrace$ with even cardinality,
let $V_\lambda$
denote the subspace of $V$
spanned by the vectors $\chi_y \in V$
for all $y \in P$ of type $\lambda$ in Definitions \ref{chiy} and \ref{type}.
Then for $\lambda \subseteq \lbrace 1, 2, \ldots, N \rbrace$ with even cardinality,
we define a matrix $E_\lambda \in \mathrm{Mat}_P(\mathbb{C})$ such that
\begin{align*}
(E_\lambda - I)V_\lambda = 0, \\
E_\lambda V_{\lambda'} = 0 && \text{if $\lambda \neq \lambda'$},
\end{align*}
where
$\lambda' \subseteq \lbrace 1, 2, \ldots, N \rbrace$ with even cardinality.
In other words,
$E_\lambda$ is the projection from $V$ onto $V_\lambda$.
Observe that 
$E_\mu^*$ and $E_\lambda$ commute for all $\mu \in \lbrace 0,1 \rbrace^N$ and $\lambda \subseteq \lbrace 1, 2, \ldots, N \rbrace$ with even cardinality.

\begin{lem}\label{lem:EmuElambda}
For $\mu \in \lbrace 0,1 \rbrace^N$ and for $\lambda \subseteq \lbrace 1, 2, \ldots, N \rbrace$ with even cardinality,
the following are equivalent:
\begin{enumerate}
\item $E_\mu^*E_\lambda = E_\lambda E_\mu^* \neq 0$;
\item the pair $(\lambda \cap S_\mu, \lambda \cap T_\mu)$ satisfies (i), (ii) in Lemma \ref{PIcond}.
\end{enumerate}
\end{lem}

\begin{proof}
This is a matrix interpretation of Lemma \ref{LSmuLTmu}.
\end{proof}

\section{The $L_mR_m$- and $R_mL_m$-actions on $V$}\label{LRRLactiononV}

In this section, we fix a basis $v_1, v_2, \ldots, v_N$ for $H$
adapted to the flag $\lbrace x_i \rbrace_{i=0}^N$
and assume that the matrix forms in Definition \ref{matrixform}
and the types in Definition \ref{type}
are always taken with respect to this basis $v_1, v_2, \ldots, v_N$.
We also fix a nontrivial character $\chi$ of the additive group $\mathbb{F}_q$.
Recall 
from Section \ref{LRactiononV},
the definition of $E_\lambda$ for $\lambda \subseteq \lbrace 1, 2, \ldots, N \rbrace$ with even cardinality
depends on the basis $v_1, v_2, \ldots, v_N$ and on the character $\chi$.
We show in this section, that $E_\lambda$ is independent of the basis $v_1, v_2, \ldots, v_N$ for $H$
adapted to the flag $\lbrace x_i \rbrace_{i=0}^N$
and the nontrivial character $\chi$ of the additive group $\mathbb{F}_q$.

\begin{lem}\label{lem:RmLmchiy}
Let $1 \le m \le N$,
and let $\mu \in \lbrace 0,1 \rbrace^N$ and $\lambda \subseteq \lbrace 1, 2, \ldots, N \rbrace$ satisfy (ii) in Lemma \ref{LSmuLTmu}.
Set
\begin{equation}
\kappa(m,\mu,\lambda) = |S_\mu(m-1) \setminus \lambda| + |T_\mu(m+1) \setminus \lambda| + |\lambda|/2.
\label{kappa}
\end{equation}
Then for $v \in E_\mu^* E_\lambda V$,
we have the following:
\[
R_mL_m v = 
\begin{cases}
q^{\kappa(m,\mu,\lambda)} v & \text{if $m \in T_\mu$ and $m \not\in \lambda$},\\
0 & \text{if $m \in S_\mu$ or $m \in \lambda$}.
\end{cases}
\]
\end{lem}

\begin{proof}
Observe that $R_mL_m$ acts on $E_\mu^* V$ by Lemma \ref{lem:actLmRm} (ii).
Fix $y \in P_\mu$ of type $\lambda$ in Definition \ref{type}.
We show that $\chi_y$ is an eigenvector for $R_mL_m$.
If $m \in S_\mu$ or $m \in \lambda$,
then by Lemma \ref{Lmchiy0},
we have $L_m\chi_y = 0$ and so 
$\chi_y$ is an eigenvector for $R_mL_m$ with respect to the eigenvalue $0$.
If $\mu = \mathbf{1}$,
then $P_\mu = \lbrace H \rbrace$ and $\lambda = \emptyset$.
So we have $\dim E_\mu^* V = 1$.
Therefore, $\chi_y$ is an eigenvector of $R_mL_m$
and the corresponding eigenvalue
is the number of subspaces which are $m$-covered by $y = H$,
which is equal to $q^{N-m} = q^{\kappa(m, \mathbf{1},\emptyset)}$
by Lemma \ref{lem:adj} (i).
Set $\nu = \mu - \widehat{m}$ so that $\mu$ $m$-covers $\nu$.
If $m \in T_\mu$, $m \not\in \lambda$ and $\nu = \mathbf{0}$, then $P_\nu = \lbrace 0 \rbrace$ and $\lambda = \emptyset$.
In other words, the matrix form of $y$ in Definition \ref{matrixform}
equals the zero matrix $O$,
and so $y'$-th entry $\chi_y(y')$ of $\chi_y$ is $1$ if $y' \in P_\mu$
and $0$ if $y' \not\in P_\mu$.
Since $P_\nu = \lbrace 0 \rbrace$,
$\chi_y$ is an eigenvector of $R_mL_m$
and the corresponding eigenvalue
is the number of subspaces which $m$-covers $z = 0$,
which is equal to $q^{m-1} = q^{\kappa(m,\widehat{m},\emptyset)}$ by Lemma \ref{lem:adj} (ii).
If $m \in T_\mu$, $m \not\in \lambda$, $\mu \neq \mathbf{1}$
and $\nu \neq \mathbf{0}$,
then
we have
\[
R_mL_m\chi_y = \frac{1}{|P_\mu|}\sum_{y' \in P_\mu} \langle R_mL_m\chi_y, \chi_{y'}\rangle \chi_{y'}.
\]
Let $y' \in P_\mu$.
Since $L_m$ and $R_m$ are (conjugate-)transposes of one another, 
we have
\begin{align*}
\langle R_mL_m\chi_y, \chi_{y'}\rangle &= \langle L_m\chi_y, L_m\chi_{y'}\rangle \\
&= \sum_{z \in P_\nu} L_m\chi_y(z) \overline{L_m\chi_{y'}(z)}.
\end{align*}
Let $Y, Y' \in \mathcal{M}_\mu(\mathbb{F}_q)$ and $Z \in \mathcal{M}_{\nu}(\mathbb{F}_q)$
be the matrix forms of $y, y', z$, respectively in Definition \ref{matrixform}.
Then by Lemma \ref{lem:Lmchiy}, it becomes
\begin{align*}
\sum_{z \in P_\nu} L_m\chi_y(z) \overline{L_m\chi_{y'}(z)} = q^{2|S_\mu(m-1)|}
\sum\chi\left(\sum_{s \in S_\mu}\sum_{t \in T_\nu} \left(Y_{s,t} - Y'_{s,t}\right)Z_{s,t}\right),
\end{align*}
where the sum is taken over all $Z \in \mathcal{M}_\nu(\mathbb{F}_q)$
such that
\begin{align}
\sum_{t \in T_\nu} Y_{s,t}Z_{m,t} = Y_{s,m},
&&
\sum_{t \in T_\nu} Y'_{s,t}Z_{m,t} = Y'_{s,m}
\label{equations}
\end{align}
for all $s \in S_\mu$ with $s < m$.
Then, since
$\mathrm{Supp}(Z) \subseteq B_\nu$,
by the orthogonality of the character $\chi$ and the trivial character,
the sum vanishes unless 
$Y_{s,t} = Y'_{s,t}$
for all $s \in S_\mu$ and $t \in T_\nu$ with $s < t$,
which by \eqref{equations} and Lemma \ref{numberofZ}
implies $Y = Y'$ and so $y = y'$.
In particular, $\chi_y$ is an eigenvector of $R_mL_m$.
Moreover, using Lemma \ref{numberofZ}
and $|P_\mu| = q^{|B_\mu|}$,
we can easily show that the corresponding eigenvalues is 
$q^{\kappa(m,\mu,\lambda)}$.
\end{proof}

\begin{lem}\label{lem:LmRmchiy}
Let $1 \le m \le N$,
and let $\mu \in \lbrace 0,1 \rbrace^N$ and $\lambda \subseteq \lbrace 1, 2, \ldots, N \rbrace$ satisfy (ii) in Lemma \ref{LSmuLTmu}.
Recall $\kappa(m,\mu,\lambda)$ from \eqref{kappa}.
Then for $v \in E_\mu^* E_\lambda V$,
we have the following:
\[
L_mR_m v = 
\begin{cases}
q^{\kappa(m,\mu,\lambda)} v & \text{if $m \in S_\mu$ and $m \not\in \lambda$},\\
0 & \text{if $m \in T_\mu$ or $m \in \lambda$}.
\end{cases}
\]
\end{lem}

\begin{proof}
Similar to the proof of Lemma \ref{lem:RmLmchiy}.
\end{proof}

\begin{prop}\label{prop:Eilambda}
For $\lambda \subseteq \lbrace 1, 2, \ldots, N \rbrace$ with even cardinality,
the matrix $E_\lambda$ belongs to the algebra $\mathcal{H}$
in Definition \ref{H}.
\end{prop}

\begin{proof}
Referring to \eqref{kappa},
we set
\[
\theta(m, \mu, \lambda)
=
\begin{cases}
q^{\kappa(m,\mu,\lambda)}
& \text{if $m \not\in \lambda$}, \\
0 & \text{if $m \in \lambda$}
\end{cases}
\]
for $1 \le m \le N$, $\mu \in \lbrace 0,1 \rbrace^N$ and $\lambda \subseteq \lbrace 1, 2, \ldots, N \rbrace$ 
satisfying (ii) in Lemma \ref{LSmuLTmu}.
Then by Lemmas \ref{lem:RmLmchiy} and \ref{lem:LmRmchiy},
we have
\[
R_mL_m + L_mR_m
= \sum_{\mu,\lambda}
\theta(m,\mu,\lambda) E_\mu^* E_\lambda,
\]
where the sum
is taken over
all pairs $\mu \in \lbrace 0,1 \rbrace^N$ and $\lambda \subseteq \lbrace 1, 2, \ldots, N \rbrace$ 
satisfying (ii) in Lemma \ref{LSmuLTmu}.
Pick $\mu \in \lbrace 0,1 \rbrace^N$ and
multiply each term on the left of the above equation,
by $E_\mu^*$.
Then we obtain
\[
E_\mu^*R_mL_m + E_\mu^*L_mR_m
= \sum_{\lambda}
\theta(m,\mu,\lambda) E_\mu^*E_{\lambda},
\]
where the sum
is taken over
$\lambda \subseteq \lbrace 1, 2, \ldots, N \rbrace$ 
satisfying (ii) in Lemma \ref{LSmuLTmu}.
For a subset $\lambda' \subseteq \lbrace 1, 2, \ldots, N \rbrace$, since $E_\mu^*, E_\lambda$ are mutually commutative and they are idempotents, we have
\begin{equation}
\prod_{m \in \lambda'}
\left(E_\mu^*R_mL_m + E_\mu^*L_mR_m\right)
= \sum_{\lambda}
\left(
\prod_{m \in \lambda'}
\theta(m,\mu,\lambda)\right) E_\mu^*E_{\lambda},
\label{eq:lambda'theta}
\end{equation}
where the sum
is taken over
$\lambda \subseteq \lbrace 1, 2, \ldots, N \rbrace$ 
satisfying (ii) in Lemma \ref{LSmuLTmu}.
Observe that the coefficient $\prod_{m \in \lambda'}
\theta(m,\mu,\lambda)$ vanishes if and only if $\lambda \cap \lambda' \neq \emptyset$.

We show that each $E_\mu^*E_{\lambda}$ is a polynomial in $E_\mu^*R_mL_m + E_\mu^*L_mR_m$ $(1 \le m \le N)$ by induction on $|\lambda|$.
If we apply $\lambda' = \lbrace 1, 2, \ldots, N \rbrace$ to the equation \eqref{eq:lambda'theta}, then
the right-hand side becomes a nonzero scalar multiple of $E_\mu^*E_{\emptyset}$.
This means that $E_\mu^*E_{\emptyset}$ is a a polynomial in $E_\mu^*R_mL_m + E_\mu^*L_mR_m$ $(1 \le m \le N)$.
Suppose each $E_\mu^*E_{\lambda''}$ is a polynomial in $E_\mu^*R_mL_m + E_\mu^*L_mR_m$ $(1 \le m \le N)$
for all $|\lambda''| < k$.
Then for $\lambda \subseteq \lbrace 1, 2, \ldots, N \rbrace$ 
with $|\lambda| = k$
satisfying (ii) in Lemma \ref{LSmuLTmu},
we apply $\lambda' = \lbrace 1, 2, \ldots, N \rbrace \setminus \lambda$ to the equation \eqref{eq:lambda'theta}.
The right-hand side is a nonzero scalar multiple of $E_\mu^*E_{\lambda}$ plus
a linear combination of $E_\mu^*E_{\lambda''}$ with $|\lambda''| < k$,
which is a polynomial in $E_\mu^*R_mL_m + E_\mu^*L_mR_m$ $(1 \le m \le N)$ by inductive hypothesis.
This means $E_\mu^*E_{\lambda}$ is also a polynomial in $E_\mu^*R_mL_m + E_\mu^*L_mR_m$ $(1 \le m \le N)$.
Therefore 
each $E_\mu^*E_{\lambda}$ is a polynomial in $E_\mu^*R_mL_m + E_\mu^*L_mR_m$ $(1 \le m \le N)$.
Observe that
for $\lambda \subseteq \lbrace 1, 2, \ldots, N \rbrace$ with even cardinality,
we have
\[
E_\lambda = \sum_\mu E_\mu^*E_{\lambda}
\]
where the sum is taken over all 
$\mu \in \lbrace 0,1 \rbrace^N$
such that the pair $(\lambda \cap S_\mu, \lambda \cap T_\mu)$ satisfies (i), (ii) in Lemma \ref{PIcond}.
Then the result follows.
\end{proof}

We remark that the above proof of Proposition \ref{prop:Eilambda}
also shows that the matrices $E_\lambda$ are independent of the basis $v_1, v_2, \ldots, v_N$ for $H$
adapted to the flag $\lbrace x_i \rbrace_{i=0}^N$
and the nontrivial character $\chi$ of the additive group $\mathbb{F}_q$.

\begin{lem}\label{lem:ker}
Let $V_{\mathrm{new}}$ denote the set of all $v \in V$ such that
$L_mv = 0$ for all $1 \le m \le N$.
Then
we have
\begin{align*}
V_{\mathrm{new}} = 
\sum_{\mu, \lambda} E_\mu^*E_\lambda V
&&\text{(direct sum)},
\end{align*}
where the sum is taken over all pairs $(\mu, \lambda)$
with $\mu \in \lbrace 0,1 \rbrace^N$ and $\lambda \subseteq \lbrace 1, 2, \ldots, N \rbrace$ 
satisfying (ii) in Lemma \ref{LSmuLTmu}
such that $\lambda$ is column-full with respect to $\mu$ in Definition \ref{columnfull}.
\end{lem}

\begin{proof}
Take $\mu \in \lbrace 0,1 \rbrace^N$ and $\lambda \subseteq \lbrace 1, 2, \ldots, N \rbrace$ 
satisfying (ii) in Lemma \ref{LSmuLTmu}.
Observe that
the following are equivalent:
\begin{enumerate}
\item for $1 \le m \le N$, we have either $m \in S_\mu$ or $m \in \lambda$;
\item $\lambda$ is column-full with respect to $\mu$.
\end{enumerate}
Then by Lemma \ref{Lmchiy0}, 
if $\lambda$ is column-full with respect to $\mu$,
we have $E_\mu^*E_\lambda V \subseteq V_{\mathrm{new}}$.
Suppose $\lambda$ is not column-full with respect to $\mu$.
Then
there exists 
$1 \le m \le N$ such that $m \in T_\mu$ and $m \not\in \lambda$.
By Lemma \ref{lem:RmLmchiy},
for any $v \in E_\mu^*E_\lambda V$,
$R_mL_mv$ is a nonzero scalar multiple of $v$. In particular,
$L_mv \neq 0$ and so $v \not\in V_{\mathrm{new}}$.
By above comments and by the fact that $V$ is the direct sum of $E_\mu^*E_\lambda V$, the result follows.
\end{proof}

Recall the column-full property in Definition \ref{columnfull}.
For $\mu \in \lbrace 0,1 \rbrace^N$ and $\lambda \subseteq \lbrace 1, 2, \ldots, N \rbrace$ satisfying (ii) in Lemma \ref{LSmuLTmu},
we say $\lambda$ is \emph{row-full} with respect to $\mu$
if $S_\mu \subseteq \lambda$.

\begin{lem}\label{lem:old}
Let $V_{\mathrm{old}}$ denote the set of all $v \in V$ such that
$R_mv = 0$ for all $1 \le m \le N$.
Then
we have
\begin{align*}
V_{\mathrm{old}} = 
\sum_{\mu, \lambda} E_\mu^*E_\lambda V
&&\text{(direct sum)},
\end{align*}
where the sum is taken over all pairs $(\mu, \lambda)$
with $\mu \in \lbrace 0,1 \rbrace^N$ and $\lambda \subseteq \lbrace 1, 2, \ldots, N \rbrace$ 
satisfying (ii) in Lemma \ref{LSmuLTmu}
such that $\lambda$ is row-full with respect to $\mu$.
\end{lem}

\begin{proof}
Similar to the proof of Lemma \ref{lem:ker}.
\end{proof}

\section{The scalar $\kappa(m,\mu,\lambda)$}\label{sec:kappa}

In this section, we discuss on the scalar $\kappa(m,\mu,\lambda)$ in \eqref{kappa}.

\begin{lem}\label{lem:kappa01}
Let $\mu \in \lbrace 0,1 \rbrace^N$ and $\lambda \subseteq \lbrace 1, 2, \ldots, N \rbrace$ 
satisfy (ii) in Lemma \ref{LSmuLTmu}.
Referring to \eqref{kappa},
we have the following.
\[
\sum_{m} (-1)^{\mu_m}\kappa(m,\mu,\lambda)
=
\frac{(N-1)(N - 2|\mu|)}{2},
\]
where the sum is taken over all $1 \le m \le N$ with $m \not\in \lambda$.
\end{lem}

\begin{proof}
Fix $\mu \in \lbrace 0,1 \rbrace^N$
and
we prove the assertion by induction on the cardinality of $\lambda$.
Let $F(\lambda)$ denote
the left-hand side of the equation.
Observe that
\begin{align*}
F(\lambda) &= \left(
\sum_{s \in S_\mu \setminus \lambda} |S_\mu(s-1) \setminus \lambda|
+
\sum_{s \in S_\mu \setminus \lambda} |T_\mu(s+1) \setminus \lambda|
+
\sum_{s \in S_\mu \setminus \lambda} \frac{|\lambda|}{2}
\right) \\
&\quad-
\left(
\sum_{t \in T_\mu \setminus \lambda} |S_\mu(t-1) \setminus \lambda|
+
\sum_{t \in T_\mu \setminus \lambda} |T_\mu(t+1) \setminus \lambda|
+
\sum_{t \in T_\mu \setminus \lambda} \frac{|\lambda|}{2}
\right).
\end{align*}
Each of the second and fourth sums counts the number of pairs $(s,t) \in S_\mu \times T_\mu$ with $s,t \not\in \lambda$ and $t > s$.
Thus,
the second and fourth terms cancel out,
i.e.,
\[
F(\lambda)=
\left(
\sum_{s \in S_\mu \setminus \lambda} |S_\mu(s-1) \setminus \lambda|
\right)
-
\left(
\sum_{t \in T_\mu \setminus \lambda} |T_\mu(t+1) \setminus \lambda|
\right)
+
\frac{|\lambda|}{2}\left(|S_\mu \setminus \lambda| - |T_\mu \setminus \lambda|\right).
\]

If $\lambda = \emptyset$
then
we have
\[
\sum_{s \in S_\mu} |S_\mu(s-1)|
=
0 + 1 + \cdots + (N-|\mu|-1)
=
\frac{(N-|\mu|)(N-|\mu|-1)}{2},
\]
and
\[
\sum_{t \in T_\mu} |T_\mu(t+1)|
=
0 + 1 + \cdots + (|\mu|-1)
=
\frac{|\mu|(|\mu|-1)}{2}.
\]
Therefore,
we have
\[
F(\emptyset) = 
\frac{(N-|\mu|)(N-|\mu|-1)}{2}
-
\frac{|\mu|(|\mu|-1)}{2}
=
\frac{(N-1)(N - 2|\mu|)}{2}
\]
and the result follows.

If $|\lambda| \ge 1$,
there exist 
$s = \max (\lambda \cap S_\mu)$ and $t = \max (\lambda \cap T_\mu)$
since 
the pair $(\lambda \cap S_\mu, \lambda \cap T_\mu)$ satisfies (i) in Lemma \ref{PIcond}.
Set
$\lambda' = \lambda \setminus \lbrace s,t \rbrace$
and
observe that
$\lambda'$ satisfies (ii) in Lemma \ref{LSmuLTmu} and we have 
\[
\sum_{s' \in S_\mu \setminus \lambda} |S_\mu(s'-1) \setminus \lambda| =
\left(
\sum_{s' \in S_\mu \setminus \lambda'} |S_\mu(s'-1) \setminus \lambda'|
\right)
-|S_\mu \setminus \lambda|,
\]
and
\[
\sum_{t' \in T_\mu \setminus \lambda} |T_\mu(t'+1) \setminus \lambda| =
\left(
\sum_{t' \in T_\mu \setminus \lambda'} |T_\mu(t'+1) \setminus \lambda'|
\right)
-|T_\mu \setminus \lambda|.
\]
Therefore,
since $|\lambda| = |\lambda'|+2$,
we have
\[
F(\lambda) = F(\lambda')
\]
and by the inductive hypothesis, the result follows.
\end{proof}

In the next lemma, we do not assume $q$ to be a prime power.

\begin{lem}\label{lem:mathfrakq}
Let $\mu = (\mu_1, \mu_2, \ldots, \mu_N) \in \lbrace 0,1 \rbrace^N$ and $\lambda \subseteq \lbrace 1, 2, \ldots, N \rbrace$ 
satisfy (ii) in Lemma \ref{LSmuLTmu}.
Referring to \eqref{kappa},
for $q \in \mathbb{C}$ with $q \neq 0,1$,
we have the following.
\[
\sum_{m} (-1)^{\mu_m}q^{\kappa(m,\mu,\lambda)}
=
\frac{q^{N-|\mu|} - q^{|\mu|}}{q-1},
\]
where the sum is taken over all $1 \le m \le N$ with $m \not\in \lambda$.
\end{lem}

\begin{proof}
For notational convenience,
in this proof we use the following notation.
Take $n \in \mathbb{N} \setminus \lbrace 0 \rbrace$.
For $\nu = (\nu_1, \nu_2, \ldots, \nu_n) \in \lbrace 0,1 \rbrace^n$,
a sequence
$\mathfrak{a} = (\mathfrak{a}_1, \mathfrak{a}_2, \ldots, \mathfrak{a}_n) \in \mathbb{Z}^n$
is called a \emph{$\kappa$-sequence with respect to $\nu$}
whenever
it satisfies
\[
\mathfrak{a}_i = \begin{cases}
\mathfrak{a}_{i-1} + 1 & \text{if $\nu_{i-1} = \nu_i$}, \\
-\mathfrak{a}_{i-1} & \text{if $\nu_{i-1} \neq \nu_i$}
\end{cases}
\]
for $2 \le i \le n$.
We call $\nu \in \lbrace 0,1 \rbrace^n$ \emph{reduced}
if $n \le 2$ or $\nu$ is either $\mathbf{0}$ or $\mathbf{1}$.
Let $\mathfrak{a} = (\mathfrak{a}_1, \mathfrak{a}_2, \ldots, \mathfrak{a}_n) \in \mathbb{Z}^n$ be 
a $\kappa$-sequence with respect to a non-reduced $\nu = (\nu_1, \nu_2, \ldots, \nu_n) \in \lbrace 0,1 \rbrace^n$.
Then 
we have $\nu_{i-1} \neq \nu_i$ for some $2 \le i \le n$.
Let $\nu' \in \lbrace 0,1 \rbrace^{n-2}$
be the sequence obtained from $\nu$ by removing 
the coordinates $i-1$ and $i$,
and let $\mathfrak{a}' \in \mathbb{Z}^{n-2}$
denote the sequence obtained from $\mathfrak{a}$ by removing 
the same pair of coordinates.
Then it is easy to show that
the sequence $\mathfrak{a}'$ is again 
a $\kappa$-sequence with respect to $\nu'$.
Moreover,
by continuing this process,
any $\kappa$-sequence reaches a $\kappa$-sequence with respect to a reduced tuple $\nu$.
More precisely,
a $\kappa$-sequence $\mathfrak{a}$ with respect to $\nu \in \lbrace 0,1 \rbrace^n$
becomes
\begin{enumerate}
\item a $\kappa$-sequence of length $2$ with respect to $(0,1)$ or $(1,0)$
if $2|\nu| = n$,
\item a $\kappa$-sequence of length $n - 2|\nu|$ with respect to $\mathbf{0} \in \lbrace 0,1 \rbrace^{n - 2|\nu|}$
if $2|\nu| < n$,
\item a $\kappa$-sequence of length $2|\nu| - n$ with respect to $\mathbf{1} \in \lbrace 0,1 \rbrace^{2|\nu| - n}$
if $2|\nu| > n$.
\end{enumerate}
We call this a \emph{reduced $\kappa$-sequence from $\mathfrak{a}$}.
For a $\kappa$-sequence $\mathfrak{a} = (\mathfrak{a}_1, \mathfrak{a}_2, \ldots, \mathfrak{a}_n) \in \mathbb{Z}^n$
with respect to $\nu = (\nu_1, \nu_2, \ldots, \nu_n) \in \lbrace 0,1 \rbrace^n$,
we define
\[
f(\nu, \mathfrak{a};q) = 
\sum_{i = 1}^n (-1)^{\nu_i}q^{(-1)^{\nu_i}\mathfrak{a}_i}.
\]
Observe that
the value $f(\nu, \mathfrak{a};q)$
is invariant under the reducing process above.
In particular,
if $\mathfrak{a}'$ is a reduced $\kappa$-sequence with respect to $\nu'$ obtained from a $\kappa$-sequence $\mathfrak{a}$ with respect to $\nu$,
then we have
$f(\nu, \mathfrak{a};q) = f(\nu', \mathfrak{a}';q)$.

Set $n = N - |\lambda|$.
Let $\nu = \nu(\mu,\lambda) \in \lbrace 0,1 \rbrace^n$ be 
the sequence obtained from $\mu$ by
removing all the coordinates indexed by $\lambda$.
Consider the sequence
$\mathfrak{a} \in \mathbb{Z}^n$ defined by
\[
\mathfrak{a} = ((-1)^{\mu_m}\kappa(m,\mu,\lambda))_{m \in \lbrace 1,2,\dots,N\rbrace \setminus \lambda},
\]
where the index $m$ increases from left to right. 
For $1 \le m < m' \le N$ with $m,m' \not\in \lambda$,
observe that
\[
\kappa(m,\mu,\lambda) - \kappa(m',\mu,\lambda) 
=
|\lbrace t \in T_\mu \setminus \lambda \mid m < t \le m' \rbrace|
-
|\lbrace s \in S_\mu \setminus \lambda \mid m \le s < m' \rbrace|.
\]
Therefore,
the sequence $\mathfrak{a}$ is a $\kappa$-sequence with respect to $\nu$.
Let $\mathfrak{a}'$ be a reduced $\kappa$-sequence with respect to $\nu'$ from $\mathfrak{a}$.
Then the left-hand side of the desired identity
becomes $f(\nu', \mathfrak{a}';q)$.

We first consider the case $2|\mu| = N$.
Then
we have $|S_\mu| = |T_\mu|$ and
so
$2|\nu| = n$
since the pair $(\lambda \cap S_\mu, \lambda \cap T_\mu)$ satisfies (i) in Lemma \ref{PIcond}.
Thus, $\mathfrak{a}'$ is a $\kappa$-sequence of length $2$ with respect to $(0,1)$ or $(1,0)$
and so
$f(\nu', \mathfrak{a}';q) = 0$ and the result follows.
We next consider the case $2|\mu| < N$.
Then
by the similar argument above,
we have $2|\nu| < n$.
Thus, $\mathfrak{a}'$ is a $\kappa$-sequence of length $n - 2|\nu|= N - 2|\mu|$ with respect to $\mathbf{0} \in \lbrace 0,1 \rbrace^{n - 2|\nu|}$.
By the definition of $\kappa$-sequence,
$\mathfrak{a}'$
is an
arithmetic sequence
with common difference $1$.
We claim that
\[
\mathfrak{a}' =
(|\mu|, |\mu|+1, \ldots, N-|\mu|-1).
\]
To show this,
since it is an arithmetic sequence, 
it suffices to show that
\[
\sum_{a' \in \mathfrak{a}'} a'
=
\frac{(N-1)(N - 2|\mu|)}{2}.
\]
This follows from Lemma \ref{lem:kappa01} since
$\sum_{a' \in \mathfrak{a}'} a' = \sum_{a \in \mathfrak{a}} a$.
For the case $2|\mu| > N$,
the proof is similar to that for the case $2|\mu| < N$.
Hence the result follows.
\end{proof}

\section{The $\mathcal{H}$-modules}\label{sec:Hmod}

Recall from Proposition \ref{semisimple}
that the algebra $\mathcal{H}$ is semisimple.
Thus the standard module $V$ is a direct sum of
irreducible $\mathcal{H}$-modules,
and every irreducible $\mathcal{H}$-module appears in $V$
up to isomorphism.
We now discuss the $\mathcal{H}$-submodules of $V$,
which from now on we call $\mathcal{H}$-modules for short.

\begin{prop}\label{prop:WWv}
Any irreducible $\mathcal{H}$-module
is generated by 
a nonzero vector $v \in V$
such that $L_mv = 0$ for all $1 \le m \le N$.
\end{prop}

\begin{proof} 
Set
$\Phi(v) = \lbrace m \mid 1 \le m \le N, L_mv \neq 0\rbrace$
for $v \in V$.
Let $W$ denote an irreducible $\mathcal{H}$-module and take a nonzero vector $w \in W$.
If $\Phi(w) = \emptyset$, then 
$L_mw = 0$ for all $1 \le m \le N$ and by the irreducibility of $W$,
the module
$W$ is generated by $w$ and so the result follows.
Suppose $\Phi(w) \neq \emptyset$.
Let $m = \min \Phi(w)$ and set $w' = L_mw \in W$.
By Proposition \ref{prop:relations1} (i) and (ii), 
we have $\Phi(w') \subsetneq \Phi(w)$.
By continuing this process at most $|\Phi(w)|$ times, 
we get a nonzero vector $v \in W$ such that
$\Phi(v) = \emptyset$.
By the same argument above, 
the assertion holds.
\end{proof}

Recall from Sections \ref{LRactiononV} and \ref{LRRLactiononV}, that there are the matrices $E_\lambda$ in $\mathcal{H}$
and that they turn out to be independent of the basis $v_1, v_2, \ldots, v_N$ for $H$ and the nontrivial character $\chi$ of the additive group $\mathbb{F}_q$.
By Lemma \ref{lem:ker} and Proposition \ref{prop:WWv},
it suffices to consider the module $\mathcal{H}v$ for $v \in \sum_{\mu, \lambda}E_\mu^* E_\lambda V$, 
where the sum is taken over all pairs $(\mu, \lambda)$
with $\mu \in \lbrace 0,1 \rbrace^N$ and $\lambda \subseteq \lbrace 1, 2, \ldots, N \rbrace$ 
satisfying (ii) in Lemma \ref{LSmuLTmu}
such that $\lambda$ is column-full with respect to $\mu$ in Definition \ref{columnfull}.

\begin{prop}\label{prop:actLR}
Let $\mu \in \lbrace 0,1 \rbrace^N$ and $\lambda \subseteq \lbrace 1, 2, \ldots, N \rbrace$ 
satisfy (ii) in Lemma \ref{LSmuLTmu}, and assume that $\lambda$ is column-full with respect to $\mu$ in Definition \ref{columnfull}.
Recall $\kappa(m,\mu,\lambda)$ in \eqref{kappa}.
For a nonzero vector $v \in E_\mu^* E_\lambda V$,
the $\mathcal{H}$-module $\mathcal{H}v$ has a basis
\begin{align}
w(\varepsilon) \in E_{\mu+\varepsilon}^*V,
&&
\varepsilon = (\varepsilon_1, \varepsilon_2, \ldots, \varepsilon_N),
\quad
\varepsilon_m =
\begin{cases}
0 & \text{if $m \in \lambda$}, \\
\text{$0$ or $1$} & \text{if $m \not\in \lambda$},
\end{cases}
\label{basis}
\end{align}
on which the generators $L_m$, $R_m$ $(1 \le m \le N)$
act as follows:
\begin{align*}
L_mw(\varepsilon) &= 
q^{\kappa(m,\mu,\lambda)-(\varepsilon_1+ \cdots + \varepsilon_{m-1})}
w(\varepsilon - \widehat{m}), \\
R_mw(\varepsilon) &= 
q^{\varepsilon_{m+1}+ \cdots + \varepsilon_N}
w(\varepsilon + \widehat{m}),
\end{align*}
where we set $w(\varepsilon) = 0$ if $\varepsilon$ is not of the form in \eqref{basis}.
\end{prop}

\begin{proof}
Let $\mathcal{H}^+$ denote the subalgebra of $\mathcal{H}$ generated by $R_1, R_2, \ldots, R_N$.
Consider $\mathcal{H}^+v$,
the $\mathcal{H}^+$-module generated by $v$.
We show that $\mathcal{H}^+v$
is an $\mathcal{H}$-module.
Let $1 \le m \le N$.
Then
$\mathcal{H}^+v$ is $R_m$-invariant 
by the construction
and $K_m$-invariant 
by
Proposition \ref{prop:relations2} (ii), (iv).
In addition, 
$\mathcal{H}^+v$ is $L_m$-invariant by
Proposition \ref{prop:relations1} (i), (iii), (iv), 
Lemma \ref{lem:LmRmchiy} and since $L_mv = 0$ by Lemma \ref{lem:ker}.
Since $\mathcal{H}$ is generated by 
$R_m$, $L_m$ and $K_m$, for $1 \le m \le N$,
$\mathcal{H}^+v$
is an $\mathcal{H}$-module.
Thus we have $\mathcal{H}^+v = \mathcal{H}v$.
By Proposition 
\ref{prop:relations1} (i), (iii),
$\mathcal{H}^+v$
is spanned by
\[
w(\varepsilon) = 
R_N^{\varepsilon_N}R_{N-1}^{\varepsilon_{N-1}} \cdots R_1^{\varepsilon_1}v,
\]
for $\varepsilon = (\varepsilon_1, \varepsilon_2, \ldots, \varepsilon_N) 
\in \lbrace 0,1 \rbrace^N$.
By Lemma \ref{lem:actLmRm} (ii),
$w(\varepsilon) \in E_{\mu+\varepsilon}^*V$.
By Lemma \ref{Rmchiy0},
$w(\varepsilon) \neq 0$
if and only if $m \in S_\mu$ and $m \not\in \lambda$ for all $1 \le m \le N$
with $\varepsilon_m = 1$.
Thus \eqref{basis} forms a basis for 
$\mathcal{H}v$.
For $1 \le m \le N$,
the $L_m$-actions on $w(\varepsilon)$
follow from Proposition \ref{prop:relations1} (iii), (iv), Lemma \ref{lem:LmRmchiy} and $L_mv = 0$.
Similarly, for $1 \le m \le N$,
the $R_m$-actions on $w(\varepsilon)$
follow from Proposition \ref{prop:relations1} (i), (iii).
The result follows.
\end{proof}

\begin{prop}\label{prop:actK}
Referring to Proposition \ref{prop:actLR},
the basis \eqref{basis} for
$\mathcal{H}v$ satisfies the following.
\begin{align*}
K_mw(\varepsilon) &= 
q^{1/2 - (\mu_m + \varepsilon_m)}
w(\varepsilon),
\end{align*}
for $1 \le m \le N$, where $\mu = (\mu_1,\mu_2,\ldots,\mu_N)$
and $\varepsilon = (\varepsilon_1,\varepsilon_2,\ldots,\varepsilon_N)$.
\end{prop}

\begin{proof}
By Proposition \ref{prop:actLR},
we have $w(\varepsilon) \in E_{\mu + \varepsilon}^*V$.
The result follows from the definition of $K_m$.
\end{proof}

\begin{thm}\label{thm:endptshape}
For any irreducible $\mathcal{H}$-module $W$,
there uniquely exist 
$\mu \in \lbrace 0,1 \rbrace^N$ and $\lambda \subseteq \lbrace 1, 2, \ldots, N \rbrace$
satisfying (ii) in Lemma \ref{LSmuLTmu} where $\lambda$ is column-full with respect to $\mu$,
such that
$W$ is generated by a nonzero vector in $E_\mu^* E_\lambda V$.
Moreover,
$W$ is determined up to isomorphism by
$\mu$ and $\lambda$.
\end{thm}

\begin{proof}
By Proposition \ref{prop:WWv},
there exists 
a nonzero vector $v \in W$
with $L_mv = 0$ for all $1 \le m \le N$
such that $W = \mathcal{H}v$.
According to the direct sum decomposition in Lemma \ref{lem:ker},
we write 
\[
v = \sum_{\mu, \lambda} E_\mu^* E_\lambda v.
\]
Since $v$ is nonzero,
there exists a pair $(\mu,\lambda)$ such
that $E_\mu^* E_\lambda v \neq 0$.
By Proposition \ref{prop:Eilambda},
$E_\mu^* E_\lambda v$ belongs to $W$
and
so
by the irreducibility of $W$,
$E_\mu^*E_\lambda v$ generates $W$.
Suppose there exist another pair $(\mu',\lambda')$ and a vector $v' \in V$ such that
$E_{\mu'}^* E_{\lambda'} v'$ also generates $W$.
Thus we have the two bases \eqref{basis} for $W$.
However,
by comparing them,
we obtain $(\mu',\lambda') = (\mu,\lambda)$
and the result follows.
\end{proof}

\begin{defi}\normalfont
Referring to Theorem \ref{thm:endptshape},
we call $\mu \in \lbrace 0,1 \rbrace^N$
the \emph{endpoint} of $W$
and $\lambda \subseteq \lbrace 1, 2, \ldots, N \rbrace$ 
the \emph{shape} of $W$.
\end{defi}

\begin{cor}
Let $\lambda \subseteq \lbrace 1, 2, \ldots, N \rbrace$ with even cardinality.
For an irreducible $\mathcal{H}$-module $W$ of shape $\lambda$,
we have
\[
\dim W = 2^{N - |\lambda|}.
\]
\end{cor}

\begin{proof}
Count the vectors in the basis \eqref{basis} for $W$.
\end{proof}

\begin{thm}
For $\mu \in \lbrace 0,1 \rbrace^N$ and $\lambda \subseteq \lbrace 1, 2, \ldots, N \rbrace$
satisfying (ii) in Lemma \ref{LSmuLTmu} where $\lambda$ is column-full with respect to $\mu$,
there exists an irreducible $\mathcal{H}$-module of endpoint $\mu$ and shape $\lambda$.
Moreover, the multiplicity in $V$ is given by
\[
q^{|B_\mu| - n(\lambda \cap S_\mu)}
\prod_{s \in \lambda \cap S_\mu} \left(q^{\rho(s,\mu,\lambda)} - 1\right),
\]
where $n(\lambda \cap S_\mu)$ is defined in \eqref{npi1}
and $\rho(s,\mu,\lambda)$ is defined in Lemma \ref{qinvsigma}.
\end{thm}

\begin{proof}
Take a nonzero vector $v \in E_\mu^* E_\lambda V$.
We show that $W = \mathcal{H}v$ is irreducible.
Consider an
irreducible $\mathcal{H}$-module decomposition of $W$
as follows.
\begin{align*}
W = W_1 + W_2 + \cdots + W_r,
&&
\text{(direct sum)}
\end{align*}
for some positive integer $r \ge 1$.
According to this decomposition,
we write $v = w_1 + w_2 + \cdots + w_r$
such that $w_n \in W_n$ $(1 \le n \le r)$.
Since this sum is direct and $v \in E_\mu^* E_\lambda W$,
we find $w_n \in E_\mu^* E_\lambda W$ for $1 \le n \le r$.
However, 
by Proposition \ref{prop:actLR},
we have $\dim E_\mu^* E_\lambda W = 1$.
Thus, all the vectors $w_n$ $(1 \le n \le r)$ are scalar multiples of $v$.
This forces $r = 1$, i.e., $W$ is irreducible.

The multiplicity of $W$ in $V$ is $\dim E_\mu^* E_\lambda V$, which is determined in Corollary \ref{cor:numberM}.
\end{proof}

\section{The quantum affine algebra $U_q(\widehat{\mathfrak{sl}}_2)$}\label{Uq(sl2hat)}

In this section, we fix a nonzero scalar $q \in \mathbb{C}$ which is not a root of unity.
For $n \in \mathbb{N}$, we define
\[
[n]_q = \frac{q^n - q^{-n}}{q-q^{-1}}.
\]
We recall the definition of $U_q(\widehat{\mathfrak{sl}}_2)$ from $\cite{Ch}$ in terms of Chevalley generators.

\begin{defi}[{\cite[Section~2]{Ch}}]\label{Uqsl2hat}\normalfont
The quantum affine algebra $U_q(\widehat{\mathfrak{sl}}_2)$ is the associative $\mathbb{C}$-algebra generated by $e_i^{\pm}, k_i, k_i^{-1}$ $(i = 0,1)$ with the relations
\begin{align}
&
k_ik_i^{-1} = k_i^{-1}k_i = 1, 
&&
k_0k_1 = k_1k_0,\label{defrel1}\\
&
k_ie_i^{\pm} = q^{\pm2}e_i^{\pm}k_i,
&&
k_ie_j^{\pm} = q^{\mp2}e_j^{\pm}k_i,  \quad i \neq j, \label{defrel2}\\
&
e_i^+e_i^- - e_i^-e_i^+ = \frac{k_i - k_i^{-1}}{q-q^{-1}}, 
&&
e_0^{\pm}e_1^{\mp} - e_1^{\mp}e_0^{\pm} = 0,\label{defrel3}
\end{align}
\begin{align}
&(e_i^{\pm})^3 e_j^{\pm} - [3]_q (e_i^{\pm})^2 e_j^{\pm} e_i^{\pm} + [3]_q e_i^{\pm} e_j^{\pm} (e_i^{\pm})^2 -  e_j^{\pm} (e_i^{\pm})^3 = 0, \quad i \neq j.\label{defrel4}
\end{align}
We call $e_i^{\pm}, k_i, k_i^{-1}$ $(i = 0,1)$ the {\it Chevalley generators} for $U_q(\widehat{\mathfrak{sl}}_2)$.
\end{defi}

It is known that the quantum affine algebra $U_q(\widehat{\mathfrak{sl}}_2)$
has the following Hopf algebra structure.
The comultiplication $\Delta$ satisfies
\begin{align*}
\Delta(e_i^+) = e_i^+ \otimes k_i + 1 \otimes e_i^+,
&&
\Delta(e_i^-) = e_i^- \otimes 1 + k_i^{-1} \otimes e_i^-,
&&
\Delta(k_i) = k_i \otimes k_i.
\end{align*}
It is also known that there exists a family of
finite-dimensional irreducible $U_q(\widehat{\mathfrak{sl}}_2)$-modules
$V_d(\alpha)$ for $d \in \mathbb{N}$, $\alpha \in \mathbb{C} \setminus \lbrace 0 \rbrace$,
where $V_d(\alpha)$ has a basis $\lbrace u_i \rbrace_{i=0}^d$ satisfying
\begin{align*}
&e_0^+ u_i = \alpha [i+1]_q u_{i+1} & (0 \le i \le d-1), && e_0^+ u_d = 0, \\
&e_1^+ u_i = [d-i+1]_q u_{i-1} & (1 \le i \le d), && e_1^+ u_0 = 0, \\
&e_0^- u_i = \alpha^{-1} [d-i+1]_q u_{i-1} & (1 \le \varepsilon \le d), && e_0^- u_0 = 0, \\
&e_1^- u_i = [i+1]_q u_{i+1} & (0 \le i \le d-1), && e_1^- u_d = 0, \\
&k_0 u_i = q^{2i-d} u_i & (0 \le i \le d),\\
&k_1 u_i = q^{d-2i} u_i & (0 \le i \le d).
\end{align*}
We call $V_d(\alpha)$ the {\it evaluation module} for $U_q(\widehat{\mathfrak{sl}}_2)$
with the {\it evaluation parameter} $\alpha$.
We recurrently define the algebra homomorphism
$\Delta^{(N)}: U_q(\widehat{\mathfrak{sl}}_2) \to \underbrace{U_q(\widehat{\mathfrak{sl}}_2) \otimes \cdots \otimes U_q(\widehat{\mathfrak{sl}}_2)}_{\text{$(N+1)$ times}}$
for $N \in \mathbb{N}$ by
\begin{align*}
\Delta^{(0)} &= \mathrm{id}, \\
\Delta^{(1)} &= \Delta, \\
\Delta^{(N)} &= (\underbrace{\mathrm{id} \otimes \cdots \otimes \mathrm{id}}_{\text{$(N-2)$ times}} \otimes \Delta) \circ \Delta^{(N-1)}
&& (N \ge 2).
\end{align*}
This algebra homomorphism
$\Delta^{(N)}$ is called the {\it $N$-fold comultiplication}.
For each $N \ge 1$,
by the $(N-1)$-fold comultiplication $\Delta^{(N-1)}$, a tensor product of $N$ evaluation modules again becomes
a $U_q(\widehat{\mathfrak{sl}}_2)$-module.
More precisely,
a tensor product $V_{d_1}(\alpha_1) \otimes \cdots \otimes V_{d_N}(\alpha_N)$
has a basis 
\begin{align}
u(\varepsilon) = u_{\varepsilon_1} \otimes \cdots \otimes u_{\varepsilon_N},
&&
0 \le \varepsilon_1 \le d_1, \quad \ldots, \quad 0 \le \varepsilon_N \le d_N,
\label{basisfortensor}
\end{align}
on which the Chevalley generators act as follows:
\begin{align}
e_0^+u(\varepsilon) &= 
\sum_{m = 1}^N 
\alpha_m [\varepsilon_m+1]_q q^{2(\varepsilon_{m+1} + \cdots + \varepsilon_N) - (d_{m+1} + \cdots + d_N)} u(\varepsilon + \widehat{m}),\label{e0+u}
\\
e_1^+u(\varepsilon) &= 
\sum_{m = 1}^N 
[d_m - \varepsilon_m + 1]_q q^{(d_{m+1} + \cdots + d_N) - 2(\varepsilon_{m+1} + \cdots + \varepsilon_N)} u(\varepsilon - \widehat{m}),\label{e1+u}
\\
e_0^-u(\varepsilon) &= 
\sum_{m = 1}^N 
\alpha_m^{-1} [d_m - \varepsilon_m + 1]_q q^{(d_1 + \cdots + d_{m-1}) - 2(\varepsilon_1 + \cdots + \varepsilon_{m-1})} u(\varepsilon - \widehat{m}),\label{e0-u}\\
e_1^-u(\varepsilon) &= 
\sum_{m = 1}^N 
[\varepsilon_m + 1]_q q^{2(\varepsilon_1 + \cdots + \varepsilon_{m-1}) - (d_1 + \cdots + d_{m-1})} u(\varepsilon +\widehat{m}),\label{e1-u}\\
k_0u(\varepsilon) &= q^{2(\varepsilon_1 + \cdots + \varepsilon_N) - (d_1 + \cdots + d_N)}u(\varepsilon), \label{k0u}\\
k_1u(\varepsilon) &= q^{(d_1 + \cdots + d_N) - 2(\varepsilon_1 + \cdots + \varepsilon_N)}u(\varepsilon),\label{k1u}
\end{align}
where $\varepsilon = (\varepsilon_1, \varepsilon_2, \ldots, \varepsilon_N) \in \mathbb{Z}^N$
and
we define $u(\varepsilon) = 0$
if $\varepsilon$ is not of the form
in \eqref{basisfortensor}.

Let $W$ denote a finite-dimensional irreducible
$U_q(\widehat{\mathfrak{sl}}_2)$-module.
By \cite[Proposition~3.2]{Ch},
there exist scalars $\epsilon_0, \epsilon_1 \in \lbrace -1,1 \rbrace$
such that
each eigenvalue of $k_i$ on $W$ is $\epsilon_i$ times an integral power of $q$
for $i = 0,1$.
The pair $(\epsilon_0, \epsilon_1)$ is called the
\emph{type} of $W$.
For each pair $\epsilon_0, \epsilon_1 \in \lbrace -1,1 \rbrace$,
there exists an algebra automorphism of $U_q(\widehat{\mathfrak{sl}}_2)$
that sends
\begin{align*}
k_i \mapsto \epsilon_i k_i,
&&
e_i^+ \mapsto \epsilon_i e_i^+,
&&
e_i^- \mapsto e_i^-,
&&
(i = 0,1).
\end{align*}
By this automorphism,
any finite-dimensional irreducible
$U_q(\widehat{\mathfrak{sl}}_2)$-module of type $(\epsilon_0, \epsilon_1)$
becomes that of type $(1,1)$.

\begin{thm}[{\cite[Theorem~4.11]{Ch}}]\label{thm:CP}
Every finite-dimensional irreducible $U_q(\widehat{\mathfrak{sl}}_2)$-module
of type $(1,1)$
is isomorphic to a tensor product of evaluation modules.
Moreover, two such tensor products are isomorphic
if and only if one is obtained from the other by permuting the factors in the tensor product.
\end{thm}

With an evaluation module $V_d(\alpha)$, we associate the set of scalars
\[
S_d(\alpha) = 
\lbrace
\alpha q^{d-1}, \alpha q^{d-3}, \ldots, \alpha q^{-d+1}
\rbrace.
\]
The set $S_d(\alpha)$ is called a {\it $q$-string} of length $d$.
Two $q$-strings $S_{d_1}(\alpha_1)$, $S_{d_2}(\alpha_2)$
are said to be in {\it general position} if one of the following occurs:
\begin{enumerate}
\item
$S_{d_1}(\alpha_1) \cup S_{d_2}(\alpha_2)$ is not a $q$-string,
\item
$S_{d_1}(\alpha_1) \subseteq S_{d_2}(\alpha_2)$
or 
$S_{d_2}(\alpha_2) \subseteq S_{d_1}(\alpha_1)$.
\end{enumerate}
Moreover, several $q$-strings 
are said to be in {\it general position} 
if every two $q$-strings are in general position.

\begin{thm}[{\cite[Theorem~4.8]{Ch}}]\label{general}
A tensor product 
of evaluation modules for $U_q(\widehat{\mathfrak{sl}}_2)$
is irreducible
if and only if
the associated $q$-strings
are in general position.
\end{thm}

\section{The algebra $\mathcal{H}$ and the quantum affine algebra $U_{q^{1/2}}(\widehat{\mathfrak{sl}}_2)$}\label{sec:main}

In this section, we get back to the subspace lattice $P$ over $\mathbb{F}_q$.
Recall the matrices $E_\lambda \in \mathcal{H}$ in Sections \ref{LRactiononV} and \ref{LRRLactiononV}.
Let  $\mu \in \lbrace 0,1 \rbrace^N$ and $\lambda \subseteq \lbrace 1, 2, \ldots, N \rbrace$ 
satisfy (ii) in Lemma \ref{LSmuLTmu}.
For $v \in E_\mu^*E_\lambda V$ and $1 \le m \le N$,
if $L_mv \neq 0$, then we have $m \in T_\mu$ and $m \not\in \lambda$ by Lemma \ref{Lmchiy0}
and so $(L_mR_m)L_mv = q^{\kappa(m,\mu,\lambda)}L_mv$
by Lemma \ref{lem:RmLmchiy}.
Therefore,
we define the matrix $(L_mR_m)^{-1}L_m$ by
\begin{equation}
(L_mR_m)^{-1}L_m v = \begin{cases}
q^{-\kappa(m,\mu,\lambda)}L_mv & \text{if $L_mv \neq 0$}, \\
0 & \text{if $L_mv = 0$}
\end{cases}\label{modifiedL}
\end{equation}
for $v \in V$.
We remark that
$(L_mR_m)^{-1}L_m$ does not mean the product of 
$(L_mR_m)^{-1}$ and $L_m$
since $L_mR_m$ is not invertible by Lemma \ref{lem:RmLmchiy}.
Similarly, we define the matrix $(R_mL_m)^{-1}R_m$ by
\begin{equation}
(R_mL_m)^{-1}R_m v = \begin{cases}
q^{-\kappa(m,\mu,\lambda)}R_mv & \text{if $R_mv \neq 0$}, \\
0 & \text{if $R_mv = 0$}
\end{cases}
\label{modifiedR}
\end{equation}
for $v \in V$.

\begin{thm}\label{support}
Let $\alpha_1,\alpha_2, \ldots, \alpha_N$ denote nonzero scalars.
The standard module $V$ supports a $U_{q^{1/2}}(\widehat{\mathfrak{sl}}_2)$-module
structure on which the Chevalley generators act as follows:
\begin{center}
\renewcommand{\arraystretch}{1.7}
\begin{tabular}{cc}
generators & actions on $V$ \\\hline\hline
$e_0^{+}$ & $q^{(1-N)/2}\sum_{m=1}^N \alpha_m R_m$ \\\hline
$e_1^{+}$ & $q^{(N-1)/2}\sum_{m=1}^N (L_mR_m)^{-1}L_m$ \\\hline
$e_0^{-}$ & $\sum_{m=1}^N \alpha_m^{-1} L_m$ \\\hline
$e_1^{-}$ & $\sum_{m=1}^N (R_mL_m)^{-1}R_m$ \\\hline
$k_0$ & $\prod_{m=1}^N K_m^{-1}$ \\\hline
$k_0^{-1}$ & $\prod_{m=1}^N K_m$ \\\hline
$k_1$ & $\prod_{m=1}^N K_m$ \\\hline
$k_1^{-1}$ & $\prod_{m=1}^N K_m^{-1}$ \\\hline
\end{tabular}
\end{center}
Here the matrices $(L_mR_m)^{-1}L_m$ and $(R_mL_m)^{-1}R_m$
are defined in \eqref{modifiedL} and in \eqref{modifiedR}, respectively.
\end{thm}

\begin{proof}
Referring to the above table, for $i = 0,1$
let $\widehat{e}_i^+, \widehat{e}_i^-, \widehat{k}_i, \widehat{k}_i^{-1}$
denote the expressions to the right of $e_i^+, e_i^-, k_i, k_i^{-1}$ respectively.
We show these elements
$\widehat{e}_i^+, \widehat{e}_i^-, \widehat{k}_i, \widehat{k}_i^{-1}$ $(i = 0,1)$
satisfy the defining relations \eqref{defrel1}--\eqref{defrel4} of $U_{q^{1/2}}(\widehat{\mathfrak{sl}}_2)$ on $V$.

We first show $\widehat{e}_i^+, \widehat{e}_i^-, \widehat{k}_i, \widehat{k}_i^{-1}$ $(i = 0,1)$ satisfy the relations except the first relation in \eqref{defrel3}.
They satisfy the relations in \eqref{defrel1} by the definitions of $\widehat{k}_i, \widehat{k}_i^{-1}$ $(i = 0,1)$.
They satisfy the first relation in \eqref{defrel2} with $i = 0$ by Proposition \ref{prop:relations2}.
They satisfy the second relation in \eqref{defrel2} with $(i,j) = (1,0)$ by Proposition \ref{prop:relations2}.
Since the other relations involve $\widehat{e}_1^+, \widehat{e}_1^-$,
we show them as follows.
Fix a nonzero vector $v \in V$. 
Then we apply both sides of each defining relation to $v$
and check the results are the same.
These elements
$\widehat{e}_i^+, \widehat{e}_i^-, \widehat{k}_i, \widehat{k}_i^{-1}$ $(i = 0,1)$ satisfy the first relation in \eqref{defrel2} with $i = 1$ by Proposition \ref{prop:relations2}.
They satisfy the second relation in \eqref{defrel2} with $(i,j) = (0,1)$ by Proposition \ref{prop:relations2}.
They satisfy the second relation in \eqref{defrel3} and the relations in \eqref{defrel4} by Proposition \ref{prop:relations1}.

It remains to show that they satisfy the first relation in \eqref{defrel3}.
Take a nonzero vector $v \in E_\mu^* E_\lambda V$ for some $\mu = (\mu_1, \mu_2, \ldots, \mu_N) \in \lbrace 0,1 \rbrace^N$, $\lambda \subseteq \lbrace 1, 2, \ldots, N\rbrace$.
By Lemmas \ref{lem:RmLmchiy} and \ref{lem:LmRmchiy},
we have
\[
\left(\widehat{e}_0^+\widehat{e}_0^- - \widehat{e}_0^-\widehat{e}_0^+\right)v
=
-
\left(q^{(1-N)/2}\sum_m (-1)^{\mu_m} q^{\kappa(m,\mu,\lambda)} \right)v,
\]
where the sum is taken over all $1 \le m \le N$ with $m \not\in \lambda$.
On the other hand,
by the definition of $K_m$,
we have
\[
\left(\frac{\widehat{k}_0 - \widehat{k}_0^{-1}}{q^{1/2}-q^{-1/2}}\right) v
=
\left(\frac{q^{|\mu|-N/2} - q^{N/2 - |\mu|}}{q^{1/2}-q^{-1/2}}\right) v
\]
By Lemma \ref{lem:mathfrakq}, it turns out that both scalars are the same
and so
$\widehat{e}_0^+, \widehat{e}_0^-, \widehat{k}_0, \widehat{k}_0^{-1}$ satisfy the first relation in \eqref{defrel3}.
Similarly, 
$\widehat{e}_1^+, \widehat{e}_1^-, \widehat{k}_1, \widehat{k}_1^{-1}$ satisfy the first relation in \eqref{defrel3}.
\end{proof}

\begin{cor}\label{hom}
Let $\alpha_1,\alpha_2, \ldots, \alpha_N$ denote nonzero scalars.
There exists an algebra homomorphism from
$U_{q^{1/2}}(\widehat{\mathfrak{sl}}_2)$ to
$\mathcal{H}$ that sends
\begin{align*}
&
e_0^{+} \mapsto q^{(1-N)/2}\sum_{m=1}^N \alpha_mR_m, &&
e_1^{+} \mapsto q^{(N-1)/2}\sum_{m=1}^N (L_mR_m)^{-1}L_m, \\
&
e_0^{-} \mapsto \sum_{m=1}^N \alpha_m^{-1} L_m, &&
e_1^{-} \mapsto \sum_{m=1}^N (R_mL_m)^{-1}R_m, \\
&
k_0 \mapsto \prod_{m=1}^N K_m^{-1}, &&
k_1 \mapsto \prod_{m=1}^N K_m.
\end{align*}
\end{cor}

\begin{proof}
Immediate from Proposition \ref{support}.
\end{proof}

The algebra homomorphism in Corollary \ref{hom} turns an $\mathcal{H}$-module into a $U_{q^{1/2}}(\widehat{\mathfrak{sl}}_2)$-module.

\begin{lem}\label{lem:Uqsl2hataction}
Let $\mu \in \lbrace 0,1 \rbrace^N$ and $\lambda \subseteq \lbrace 1, 2, \ldots, N \rbrace$ 
satisfy (ii) in Lemma \ref{LSmuLTmu} where $\lambda$ is column-full with respect to $\mu$ in Definition \ref{columnfull}.
Let $W_{\mu, \lambda}$ denote an irreducible $\mathcal{H}$-module with endpoint $\mu$ and shape $\lambda$.
The basis \eqref{basis} for $W_{\mu, \lambda}$
has the following actions of Chevalley generators
via the algebra homomorphism in Corollary \ref{hom}.
\begin{align}
e_0^+ w(\varepsilon)
&=
q^{(1-N)/2}\sum_{m=1}^N \alpha_m
q^{\varepsilon_{m+1}+ \cdots + \varepsilon_N}
w(\varepsilon + \widehat{m}), \label{e0+w}\\
e_1^+ w(\varepsilon)
&=
q^{(N-1)/2}\sum_{m=1}^N
q^{-(\varepsilon_{m+1}+ \cdots + \varepsilon_N)}
w(\varepsilon - \widehat{m}),\label{e1+w} \\
e_0^- w(\varepsilon)
&=
\sum_{m=1}^N \alpha_m^{-1}
\theta_m(\mu,\lambda)
q^{-(\varepsilon_1+ \cdots + \varepsilon_{m-1})}
w(\varepsilon - \widehat{m}), \label{e0-w}\\
e_1^- w(\varepsilon)
&=
\sum_{m=1}^N
\theta_m(\mu,\lambda)^{-1}
q^{\varepsilon_1+ \cdots + \varepsilon_{m-1}}
w(\varepsilon + \widehat{m}),\label{e1-w} \\
k_0 w(\varepsilon)
&= q^{-N/2 + |\mu| + |\varepsilon|}w(\varepsilon),\label{k0w}\\
k_1 w(\varepsilon)
&= q^{N/2 - |\mu| - |\varepsilon|}w(\varepsilon),\label{k1w}
\end{align}
where $\varepsilon = (\varepsilon_1, \varepsilon_2, \ldots, \varepsilon_N) \in \lbrace 0,1 \rbrace^N$.
Here
we define $w(\varepsilon) = 0$
if $\varepsilon$ is not of the form
in \eqref{basis}.
\end{lem}

\begin{proof}
Use Propositions \ref{prop:actLR}, \ref{prop:actK} and Corollary \ref{hom}.
\end{proof}

\begin{lem}\label{lem:dm}
Let $\mu \in \lbrace 0,1 \rbrace^N$ and $\lambda \subseteq \lbrace 1, 2, \ldots, N \rbrace$ 
satisfy (ii) in Lemma \ref{LSmuLTmu} where $\lambda$ is column-full with respect to $\mu$ in Definition \ref{columnfull}.
We define 
$d = (d_1,d_2,\ldots,d_N) \in \lbrace 0,1 \rbrace^N$ by
\begin{align*}
d_m = \begin{cases}
1 & \text{if $m \not\in \lambda$}, \\
0 & \text{if $m \in \lambda$},
\end{cases}
&&
(1 \le m \le N).
\end{align*}
Then we have the following.
\begin{enumerate}
\item $|d| = N - 2|\mu|$.
\item If $m \not\in \lambda$, then $\kappa(m, \mu, \lambda) = (N-1)/2 + (d_1 + \cdots + d_{m-1})/2 - (d_{m+1} + \cdots + d_N)/2$ defined in \eqref{kappa}.
\end{enumerate}
\end{lem}

\begin{proof}
(i)
By the definition of $d$, we have $|d| = N - |\lambda|$.
By the assumption, we have $|\lambda| = 2|\mu|$ and so
the result follows.

(ii)
Assume $m \not\in \lambda$.
Observe that
\begin{align*}
|S_\mu(m-1) \setminus \lambda| = d_1 + \cdots + d_{m-1},
&&
|T_\mu(m+1) \setminus \lambda| = 0.
\end{align*}
By the definition of $d$, 
\[
|\lambda|/2 = N/2 - (d_1 + \cdots + d_N)/2.
\]
Hence the result follows from the above comments and $d_m = 1$.
\end{proof}

\begin{thm}\label{isomorphic}
Let $\mu \in \lbrace 0,1 \rbrace^N$ and $\lambda \subseteq \lbrace 1, 2, \ldots, N \rbrace$
satisfy (ii) in Lemma \ref{LSmuLTmu} where $\lambda$ is column-full with respect to $\mu$ in Definition \ref{columnfull}.
Let $W_{\mu, \lambda}$ denote an irreducible $\mathcal{H}$-module with endpoint $\mu$ and shape $\lambda$.
Then by the algebra homomorphism in Corollary \ref{hom},
$W_{\mu, \lambda}$ becomes a $U_{q^{1/2}}(\widehat{\mathfrak{sl}}_2)$-module
and we have the following.
\begin{enumerate}
\item $W_{\mu, \lambda}$ has type $(1,1)$.
\item $W_{\mu, \lambda}$ is isomorphic to
the tensor product of $V_1(\alpha_m)$,
where $1 \le m \le N$ such that $m \not\in \lambda$.
\end{enumerate}
\end{thm}

\begin{proof}
(i)
This follows from \eqref{k0w} and \eqref{k1w}.

(ii)
Recall $(d_1,d_2,\ldots,d_N) \in \lbrace 0,1 \rbrace^N$ from Lemma \ref{lem:dm}.
It suffices to show that
\[
W_{\mu, \lambda}
\simeq
V_{d_1}(\alpha_1) \otimes \cdots \otimes V_{d_N}(\alpha_N).
\]
Recall the basis $w(\varepsilon)$ in \eqref{basis} for $W_{\mu, \lambda}$
and the basis $u(\varepsilon)$ in \eqref{basisfortensor}
for $V_{d_1}(\alpha_1) \otimes \cdots \otimes V_{d_N}(\alpha_N)$,
where
$\varepsilon = (\varepsilon_1, \varepsilon_2, \ldots, \varepsilon_N) \in \lbrace 0,1 \rbrace^N$
such that
$w(\varepsilon) = 0$ and $u(\varepsilon) = 0$ 
if $d_m < \varepsilon_m$ for some $1 \le m \le N$.
We define a linear map
$\varphi$ from $V_{d_1}(\alpha_1) \otimes \cdots \otimes V_{d_N}(\alpha_N)$ to $W_{\mu, \lambda}$
that sends
$u(\varepsilon)$ to $\gamma(\varepsilon)w(\varepsilon)$,
where
\[
\gamma(\varepsilon) = 
q^{|\varepsilon|(1-N)/2}
\prod_{m \in T_\varepsilon} q^{(d_{m+1} + \cdots + d_N)/2}.
\]
We check $\varphi$ preserves the actions of Chevalley generators.
Observe that
\begin{equation}
\gamma(\varepsilon) = q^{(N-1)/2}q^{-(d_{m+1} + \cdots + d_N)/2}\gamma(\varepsilon + \widehat{m})
\label{eq:gamma}
\end{equation}
for $\varepsilon \in \lbrace 0,1 \rbrace^N$.

By \eqref{k0u} and  \eqref{k0w} and Lemma \ref{lem:dm} (i),
$\varphi$ preserves the action of $k_0$.
By \eqref{k1u} and  \eqref{k1w} and Lemma \ref{lem:dm} (i),
$\varphi$ preserves the action of $k_1$.
By \eqref{e0+u}, \eqref{e0+w} and \eqref{eq:gamma},
the map $\varphi$ preserves the action of $e_0^+$.
By \eqref{e1+u}, \eqref{e1+w} and \eqref{eq:gamma},
the map $\varphi$ preserves the action of $e_1^+$.
By \eqref{e0-u}, \eqref{e0-w}, \eqref{eq:gamma} and Lemma \ref{lem:dm} (ii),
the map $\varphi$ preserves the action of $e_0^-$.
By \eqref{e1-u}, \eqref{e1-w}, \eqref{eq:gamma} and Lemma \ref{lem:dm} (ii),
the map $\varphi$ preserves the action of $e_1^-$.
\end{proof}

\section*{Acknowledgments}
The author thanks his advisor, Hajime Tanaka,
for many valuable discussions and comments.
The author also thanks Paul Terwilliger
for giving valuable comments.
Finally, the author thanks the anonymous referee for valuable comments and useful suggestions.

\bigskip

\noindent
Yuta Watanabe \\
Graduate School of Information Sciences \\
Tohoku University \\
Sendai, 980-8579 Japan\\
email: \texttt{watanabe@ims.is.tohoku.ac.jp}

\bigskip

\noindent
Present Address:\\
Department of Mathematics\\
National Institute of Technology, Ube College\\
Ube, 755-8555 Japan\\
email: \texttt{ywatanabe@ube-k.ac.jp}

\end{document}